\documentclass[11pt]{article}

\usepackage{amssymb,amsmath,bm}
\usepackage{amsthm}
\usepackage{hyperref}
\usepackage{mathrsfs}
\usepackage{textcomp}
\usepackage{enumerate}
\usepackage{graphicx}

\newcommand{\R}{\mathbb R}
\newcommand{\T}{\mathcal T}

\def\H{\mathbb H}

\newcommand{\co}{\colon\thinspace}

\newtheorem{theorem}{Theorem}[section]
\newtheorem{lemma}[theorem]{Lemma}
\newtheorem{proposition}[theorem]{Proposition}
\newtheorem{definition}[theorem]{Definition}
\newtheorem{corollary}[theorem]{Corollary}

\theoremstyle{remark}
\newtheorem{remark}[theorem]{Remark}

\theoremstyle{remark}

\theoremstyle{remark}

\numberwithin{equation}{section}

\begin{document}

\title{\textbf{Volume and rigidity of hyperbolic polyhedral $3$-manifolds}}

\author{\medskip Feng Luo \& Tian Yang}

\date{}
\maketitle

\begin{abstract}We investigate the rigidity of hyperbolic cone metrics on $3$-manifolds which are isometric gluing of ideal and hyper-ideal tetrahedra in hyperbolic spaces. These metrics will be called ideal and hyper-ideal hyperbolic polyhedral metrics. It is shown that a hyper-ideal hyperbolic polyhedral metric is determined up to isometry by its curvature and a decorated ideal hyperbolic polyhedral metric is determined up to isometry and change of decorations by its curvature.  The main tool used in the proof is the Fenchel dual of the volume function.
\end{abstract}

\section{Introduction}

\subsection{Statements of results}

We study geometry of 3-dimensional spaces which are isometric
gluing of (ideal and hyper-ideal) tetrahedra in hyperbolic spaces.
Our main focus is on the rigidity of these spaces. The metrics of
these spaces are given by the lengths of edges of tetrahedra (in
the underlying triangulation). The curvatures of the spaces are
$2\pi$ less the cone angles at the edges. Our main results state
that for a fixed triangulation, the curvature determines the edge
lengths and hence these hyperbolic polyhedral metrics.

The  tool used in the proof is a variational principle associated
to the Schlaefli formula and its Legendre transformation. The
infinitesimal rigidity of hyperbolic cone metrics follows from the
strict convexity of the volume of the ideal and hyper-idea
tetrahedra in terms of the dihedral angles. In the dual setting,
one considers the co-volume which is the dual of the volume of the
tetrahedra and has the edge lengths as the variables. The main
difficult comes from the fact that the space of hyperbolic
tetrahedra parametrized by the edge lengths is not convex. We
overcome the difficulty by showing that the co-volume function (of
the edge lengths) can be extended to a $C^1$ smooth convex
function defined on a convex open set. This is very similar to the
results established in \cite{Luo3}. By establishing the convex
extensions of the co-volume functions, we are able to prove
several results on the volume optimization program of Casson and
Rivin. For instance, we show that the maximum volume angle
structures are exactly those coming from the generalized
polyhedral metrics (see theorem 1.3).

We now state our results more precisely.  Suppose $(M, \T)$ is a
triangulated compact pseudo 3-manifold with a triangulation $\T$
and the set of edges $E=E(\T).$

\begin{definition}
 A
{\em decorated hyperbolic polyhedral metric} (respectively {\em
hyper-ideal polyhedral metric)} on $(M, \T)$ is obtained by
replacing each tetrahedra in $\T$ by an decorated ideal
tetrahedron (respectively hyper-ideal tetrahedron) and replacing
the affine gluing homeomorphisms by isometries preserving the
decoration. The {\em curvature} of the metric assign each edge $e$
$2\pi$ less the cone angle at $e$ for interior edge $e$ and $\pi$
less the cone angle for boundary edge $e.$
\end{definition}

By the construction, these polyhedral metrics are determined by
the lengths of the edges.

\begin{theorem}\label{main} Suppose $(M, \T)$ is a triangulated
compact pseudo 3-manifold $(M, \T).$
\begin{enumerate}[(a)]
\item A decorated hyperbolic polyhedral metric on $(M, \T)$
is determined up to isometry and change of decorations by its
curvature.
\item A hyper-ideal hyperbolic
polyhedral metric on $(M, \T)$ is determined up to isometry by its
curvature.
\end{enumerate}
\end{theorem}

The rigidity results are closely related to the volume
optimization of angle structures initiated by  Casson and Rivin.
The program tries to find complete hyperbolic metrics on ideal
triangulated 3-manifolds $(N, \T)$ using \it angle structures \rm
(see \cite{La1}, \cite{Riv2}). Recall that a non-negative
(respectively positive) angle structure $\alpha$ on $(N, \T)$
assigns each edge in a tetrahedron a non-negative (respectively
positive) number called the angle so that the sum of angles around
each edge is $2\pi$ and the sum of angles at three edges from each
vertex of each tetrahedron is $\pi.$ The \it volume \rm of an
angle structure $\alpha$ is well defined using the Lobachevsky
function (see \S2.1). It is well known there are maximum volume
non-negative angle structures. If the maximum volume angle
structure  is positive, then it is known \cite{GF, Casson, Riv2} that there exists a geometric triangulation of a
complete hyperbolic metric on the manifold $N-\{vertices\}$
realizing the angle structure. Our result gives a characterization
of maximum angle structures in the case some angles are 0.

\begin{theorem}\label{1.3} Suppose $(N, \T)$ is a
triangulated closed pseudo 3-manifold which supports a positive
angle structure and $\alpha$ is a non-negative angle structure
which maximizes volume in the space of all non-negative angle
structures on $(N, \T).$ Then there exists an assignment of real number $l(e)$ to each edge $e$ so that for each
tetrahedron $\sigma \in \T,$

(1) if all angles of $\sigma$ in $\alpha$ are positive, then
$\alpha$ are the dihedral angles of the decorated ideal
tetrahedron of edge lengths given by $l$ and,

(2) if one angle of $\sigma$ in $\alpha$ is 0, then all angles of
$\sigma$ in $\alpha$ are $0,0,0,0,\pi, \pi$ and their edge lengths
in $l$ satisfy  \begin{equation}\label{1} e^{\frac{l_1+l_4}{2}}
\geqslant e^{\frac{l_2+l_5}{2}}+
e^{\frac{l_3+l_6}{2}}\end{equation} where $l_1$ is the length of
the edge of angle $\pi$ and $l_i$ and $l_{i+3}$ are lengths of
opposite edges.

Conversely, if $l: E \to \R$ is any function so that (1) and (2)
hold, then the corresponding angle $\alpha$ of $l$ defined by (1)
and (2) maximizes volume.

\end{theorem}


In section 6, we introduce the corresponding notion of
non-negative and positive angle structures of hyper-ideal type
(see Definition \ref{angle-hyper}), and prove the following
counterpart of theorem 1.3.

\begin{theorem}\label{1.4} Suppose $(N, \T)$ is a
triangulated closed pseudo 3-manifold which supports a positive
angle structure of hyper-ideal type and $\alpha$ maximizes the volume in the space of all non-negative angle
structures of hyper-ideal type on $(N, \T).$ Then there exists an assignment of
positive number $l(e)$ to each edge $e$ so that for each
tetrahedron $\sigma \in \T,$

(1) if all angles of $\sigma$ in $\alpha$ are positive, then
$\alpha$ are the dihedral angles of the hyper-ideal
tetrahedron of edge lengths given by $l$ and,

(2) if one angle of $\sigma$ in $\alpha$ is 0, then all angles of
$\sigma$ in $\alpha$ are $0,0,0,0,\pi, \pi$ and the numbers assigned
by $l$ to the edges of $\sigma$ are not the edge lengths of any hyper-ideal tetrahedron.

Conversely, if $l: E \to \R_{>0}$ is any function so that (1) and (2)
holds, then the corresponding angle $\alpha$ of $l$ defined by (1)
and (2) maximizes volume.

\end{theorem}
There have been many important work on rigidity of hyperbolic cone
metrics on 3-manifolds. See work of Hodgeson-Kerckhoff\,\cite{HK},
Weiss\,\cite{W}, Mazzeo-Montcouquiol\,\cite{MM, M} and Fillastre-Izmestiev\,\cite{FI, FI2, I} and others.
The difference between their work and ours is that we consider the
case where the singularity consists of complete geodesics from
cusp to cusp or geodesics orthogonal to the totally geodesic
boundary with possible cone singularities.

The paper is organized as follows. In Sections 2, we
 collect preliminary materials including decorated hyperbolic tetrahedra, angle structures,
 volume functions and the Fenchel dual. In Section 3, we define the co-volume
 function and reveal its relationship with the Fenchel dual of the volume function. As
  a consequence, we prove Theorem \ref{main} (a) and Theorem \ref{1.3}. The second part
  of the paper focuses on polyhedral metrics and angle structures of hyper-ideal type.
  The corresponding preliminary materials are included in Section 4, and Theorem \ref{main} (b) and
  Theorem \ref{1.4} are respectively proved in Section 5 and Section 6.

The work is partially supported by the NSF DMS 1105808 and NSF DMS
1222663.  We would like to thank Steven Kerckhoff for discussions.


\section{Preliminaries on triangulations,
volume and Fenchel duality} \label{decorated}


Since this paper involves topological triangulations, geometry of
tetrahedra in hyperbolic space $\H^3$ and convex optimization, we
will briefly recall the related material in this section.

\subsection{Triangulations}
Take a finite disjoint collection $T$ of  Euclidean tetrahedra and
identify some of the codimension-1 faces in $T$ in pairs by affine
homeomorphisms. The quotient space $(M, \T)$ is a compact pseudo
3-manifold $M$ together with a triangulation $\T$ whose simplexes
are the quotients of simplexes in $T.$ If each codimension-1 face
of $T$ are identified with another codimension-1 face, then $M$ is
a closed pseudo 3-manifold. Otherwise, $M$ is a compact pseudo
3-manifold with non-empty boundary which is the quotient of the
union of un-identified codimension-1 faces in $T.$

Two edges of tetrahedra in $T$ are called equivalent if they are
mapped to the same set in $M.$ We define \it edges \rm in the
triangulation $\T$ to be equivalence classes of edges in
tetrahedra in $T.$ We use $E =E(\T)$ and $T(\T)$ to denote the
sets of all edges and tetrahedra in $\T$ respectively. Since a
tetrahedron in the triangulation $\T$ is the same as a tetrahedron
in the original set $T,$ we will identify $T(\T)$ with the set
$T.$ A \it quad \rm in the triangulation $\T$ is a pair of
opposite edges in $T.$ Thus each tetrahedron contains three quads.
We use $\Box=\Box(\T)$ to denote the set of all quads in $\T.$
If $q \in \Box,$ $e\in E$ and $\sigma \in T,$ we  use $q \subset
\sigma$ to denote that the quad $q$ is contained in the
tetrahedron $\sigma$ and use $q \sim e$ or $e \sim q$ to denote
that $q \cap e =\emptyset$ and there exists $\sigma \in T$ with $q \subset \sigma$ and $\sigma \cap e \neq \emptyset$.

Using these notations and the fact that angles at opposite edges
in a tetrahedron are the same for any angle structure, a \it
non-negative angle structure \rm on a closed pseudo 3-manifold
$(M, \T)$ is a map $x: \Box \to \R_{\geqslant 0}$ so that (1) $\forall
\sigma \in T,$ $\sum_{ q \subset \sigma} x(q)=\pi$ and (2)
$\forall e \in E,$ $\sum_{ q \sim e} x(q) =2\pi.$
 See for
instance \cite{Luo5}  for more details.

\subsection{Decorated ideal tetrahedra in the hyperbolic 3-space}

An \it ideal n-simplex $s$ \rm in the hyperbolic n-space $\H^n$ is
the convex hull of n+1 points $v_1, ..., v_{n+1}$ in $\partial
\H^n$ so that $\{v_1, ..., v_{n+1}\}$ are not in a round
$(n-1)$-sphere. Any two ideal triangles are isometric. An ideal
tetrahedron in $\H^3$ is determined up to isometry by its six
dihedral angles. These angles satisfy the condition that angles at
opposite edges are the same and the sum of all angles is $2\pi.$
Thus the space of all ideal tetrahedra modulo isometry can be
identified with $\mathcal A=\{(a,b,c) \in \R^3_{>0}|a+b+c=\pi\}.$

Following Penner \cite{P}, a \it decorated ideal n-simplex \rm (or
simply decorated simplex) is a pair $(s,$$ \{H_1, $ .$.$.,$
H_{n+1}\})$ where $s$ is an ideal n-simplex and $H_i$ is an
$(n-1)$-horosphere centered at the i-th vertex $v_i.$  We call
$\{H_1, ..., H_{n+1}\}$ the \it decoration\rm. Two decorated
simplexes are \it equivalent \rm if there is a decoration
preserving isometry between the underlying ideal simplexes. Each
edge $e=v_iv_j$ in a decorated simplex has the signed \it length
\rm $l_{ij}$ defined as follows. The absolute value $|l_{ij}|$ of
the length is the distance between $H_i\cap e$ and $H_j \cap e$ so
that $l_{ij}>0$ if $H_i$ and $H_j$ are disjoint and $l_{ij}\leqslant 0$
if $H_i \cap H_j \neq \emptyset.$ Faces of decorated ideal
simplexes are decorated ideal simplexes. Also $s \cap H_i$ is
isometric to a Euclidean $(n-1)$-simplex.

For a decorated ideal triangle $(s, \{H_1, H_2, H_3\}),$ we call
the length $a_i$ of the horocyclic arc in $H_i$ bounded by the two
edges of $s$ from $v_i$ the \it angle \rm at $v_i.$ Penner's
cosine law says that for $\{i,j,k\}=\{1,2,3\},$
\begin{equation}\label{cosine} a_i
=e^{\frac{l_{jk}-l_{ij}-l_{ik}}{2}}. 
\end{equation}
Given any three real numbers $l_2, l_2, l_3,$ there exists a
unique decorated ideal triangle whose lengths are $l_1, l_2, l_3.$
See \cite{P}.

The characterization of the lengths of decorated ideal tetrahedron
is well known (see for instance \cite{Luo4} lemma 2.5, or
\cite{BPS} lemma 4.2.3).

\noindent
\begin{lemma}\label{L1}
Suppose $\{l_{ij}\}$ are the edge lengths of a decorated ideal
tetrahedron $\sigma =(s,\{H_1, ..., H_4\}).$ Then all four
Euclidean triangles $\{H_i\cap s\}$ are similar to the Euclidean
triangle $\tau$ of edge lengths $e^{\frac{l_{ij}+l_{kh}}{2}},$ so
that
$$e^{\frac{l_{ij}+l_{kh}}{2}}+e^{\frac{l_{ik}+l_{jh}}{2}}>e^{\frac{l_{ih}+l_{jk}}{2}}$$
for $\{i,j,k,h\}=\{1,2,3,4\}.$ The dihedral angle $\alpha_{ij}$ of
$\sigma$ at the edge $v_iv_j$ is equal to the inner angle of
$\tau$ opposite to the edge of length
$e^{\frac{l_{ij}+l_{kh}}{2}}.$ Conversely, if
$(l_{12},\dots,l_{34})\in\mathbb R^6$ satisfies the triangular
inequalities above,
 then there is a unique decorated ideal tetrahedron having $l_{ij}$ as the length at  the edge $v_iv_j.$
\end{lemma}



One consequence of the lemma is that dihedral angles
$\alpha_{ij}=\alpha_{kl},$ i.e., dihedral angles at opposite edges
are the same. Thus, we can talk about the dihedral angle of a quad
in a decorated simplex.





\subsection{Generalized decorated tetrahedra, dihedral angles and
volume}

A \it generalized decorated tetrahedron \rm is a (topological)
3-simplex of vertices $v_1, ..., v_4$ so that each edge $v_iv_j$
is assigned a real number $l_{ij}=l_{ji},$ called the \it
length\rm.  A decorated ideal tetrahedron (with the signed edge
lengths) is a generalized decorated tetrahedron. The space of all
generalized decorated tetrahedra parameterized by the length
vectors $l=(l_{12}, ..., l_{34})$ is $\R^6.$ The subspace of all
(equivalence classes) of decorated ideal tetrahedra 
is given by $\{ (l_{12},\dots,l_{34})\in\mathbb R^6\ | \
e^{\frac{l_{ij}+l_{kh}}{2}}+e^{\frac{l_{ik}+l_{jh}}{2}}
>e^{\frac{l_{ih}+l_{jk}}{2}},  \text{ \{i,j,k,h\} distinct}\}.$


To define \it dihedral angles \rm and \it volume \rm of a
generalized decorated tetrahedron $\sigma,$ let us begin with the
notion of \it generalized Euclidean triangles \rm and  their \it
angles\rm. A \it generalized Euclidean triangle \rm $\Delta$ is a
(topological) triangle of vertices $v_1, v_2, v_3$ so that each
edge is assigned a positive number, called edge length.  Let $x_i$
be the assigned length of the edge $v_jv_k$ where
$\{i,j,k\}=\{1,2,3\}.$ The \it inner angle \rm $a_i=a_i(x_1, x_2,
x_3)$ at the vertex $v_i$ is defined as follows. If $x_1,x_2, x_3$
satisfy the triangle inequalities that $x_j+x_k>x_h$ for
$\{h,j,k\}=\{1,2,3\},$ then $a_i$ is the inner angle of the
Euclidean triangle of edge lengths $x_1, x_2, x_3$ opposite to the
edge of length $x_i$; if $x_i\geqslant x_j+x_k,$ then $a_i=\pi,$
$a_j=a_k=0.$ It is known (see for instance \cite{Luo3}) that

\begin{lemma}\label{lemma2.2} The angle function $a_i(x_1,x_2, x_3): \R^3_{>0} \to [0, \pi]$ is continuous so that
$a_1+a_2+a_3=\pi$ and the $C^0$-smooth differential 1-form
$\sum_{i=1}^3 a_i d(\ln x_i)$ is closed on $\R^3_{>0}.$
Furthermore, for $u_i=\ln x_i,$ the integral $F(u)=\int^u_0
\sum_{i=1}^3 a_i(u) du_i$ is a $C^1$-smooth convex function in
$(u_1, u_2, u_3)$ on $\R^3$ so that $F$ is strictly convex when
restricted to $\{u \in \R^3| u_1+u_2+u_3=0, e^{u_i}+e^{u_j} >
e^{u_k}\}$ and $F(u+(k,k,k))=F(u)+k\pi$ for all $k \in \R$.

\end{lemma}
Here a $C^0$-smooth 1-form is defined to be closed if its
integration over any $C^1$-smooth null homotopic loop is zero.

For a generalized decorated tetrahedron of length vector
$l=(l_{12}, ..., l_{34}) \in \R^6,$ the \it dihedral angle \rm
$\alpha_{ij}=\alpha_{ij}(l)$ at the edge $v_iv_j$  is defined to
be the inner angle of the generalized Euclidean triangle of edge
lengths $e^{\frac{l_{ij}+l_{hk}}{2}},$ $e^{
\frac{l_{ij}+l_{jh}}{2}}$ and $e^{\frac{l_{ih}+l_{jk}}{2}}$ so
that $\alpha_{ij}$ is opposite to the edge of length
$e^{\frac{l_{ij}+l_{hk}}{2}}$ for $h,i,j,k$ distinct.
In particular, the dihedral angles at opposite edges are the same
and the total sum of all six dihedral angles are $2\pi.$  If
$\sigma$ is a decorated ideal tetrahedron, then the two
definitions of dihedral angles coincide.
Recall that  the Lobachevsky function is defined by
$\Lambda(x)=-\int_0^x\ln|2\sin t|dt.$  It is a continuous function
of period $\pi$ so that $\Lambda(-x)=-\Lambda(x).$ The \it volume
\rm of a generalized decorated simplex $\sigma$ of lengths
$l_{ij}$'s, denoted by $vol(l),$ is defined to be
$\frac{1}{2}\sum_{i < j} \Lambda(\alpha_{ij}(l)).$ If the
generalized decorated tetrahedron $\sigma$ is a decorated ideal
tetrahedron, then Lobachevsky showed that $vol(l)$ is the
hyperbolic volume of the underlying ideal tetrahedron. If $\sigma$
is not a decorated ideal tetrahedron, then by definition,
$vol(l)=\Lambda(0)+\Lambda(0)+\Lambda(\pi)=0.$


Since the space of all ideal tetrahedra can be parameterized by
$\mathcal A=\{(a,b,c)\in\mathbb R^3_{>0}\ |\ a+b+c=\pi\},$ the
volume function $vol$ defined on $\mathcal A$  is given by
$vol(a,b,c) =\Lambda(a)+\Lambda(b)+\Lambda(c).$ By lemma
\ref{lemma2.2}, we have



\begin{lemma}\label{PP2}

(1) The function $\alpha_{ij}: \mathbb R^6\rightarrow\mathbb R$ is
continuous.

(2)(Rivin) The volume function $vol: \mathcal A=\{(x,y,z)
\in\mathbb R^3_{>0}\ |\ x+y+z=\pi\} \to \R$ is smooth strictly
concave and extends continuously to the closure
$\overline{\mathcal A}=\{(x,y,z) \in \R_{ \geqslant 0}^3 |x+y+z=\pi\}$
so that $vol(x,y,z)=0$ if one of $x,y,z$ is 0.

(3) The volume function on the space of generalized decorated tetrahedra
$\R^6,$ $vol: \R^6 \to [0, \infty),$ is continuous.
\end{lemma}

\subsection{Fenchel Duality and hyperbolic volume}
Recall that a \it proper \rm convex function $f: \R^n \to
(-\infty, \infty]$ is a convex function so that $f(a) \neq \infty$
for some $a \in \R^n.$ A function $f:\R^n \to (-\infty, \infty]$ is
called \it lower semi-continuous \rm if for all $a \in \R^n,$
$\liminf_{ x \to a} f(x) \geqslant f(a).$  If $X \subset \R^n$ is a
non-empty closed convex set and $g:X \to \R$ is a convex lower
semi-continuous function, then the new function $\phi_g:\R^n \to
(-\infty, \infty]$ defined by $\phi_g|X=g$ and $\phi_g(x)=\infty$
for $x \notin X$ is a proper lower semi-continuous convex
function. 
 See for instance the classical book
\cite{roc} for details.

If $u,v \in \R^n,$ we use $u \cdot v$ or $<u,v>$ to denote the
standard inner product of $u$ and $v.$

\begin{definition}
The Fenchel dual $f^*: \R^n \to (-\infty, \infty]$ of a proper
function $f: \R^n \to (-\infty, \infty]$ is
$$ f^*(y) =\sup\{ x \cdot y-f(x) | x \in \R^n\}.$$
\end{definition}

It is known  that if $f$ is a proper convex function then $f^*$ is
a proper convex lower semi-continuous function. A fundamental fact
about $f^*$ is the following (see \cite{roc}),

\begin{theorem}[Fenchel] If $f: \R^n \to (-\infty, \infty]$ is a
 proper lower semi continuous convex function, then $(f^*)^* =f.$
 \end{theorem}


Since the volume function $vol:\overline{ \mathcal A}=\{(a,b,c)
\in \R^3_{\geqslant 0}| a+b+c=\pi\}$ is concave and continuous, we
obtain a proper lower semi continuous convex function $\phi: \R^3
\to (-\infty, \infty]$ defined by $\phi(x) = -vol(x)$ if $x \in
\overline{ \mathcal A}$ and $\phi(x) =\infty$ if $x \notin
\overline{ \mathcal A}.$

\begin{proposition} \label{prop2.6} The Fenchel dual $\phi^*$ of $\phi$ is
 the $C^1$-smooth convex function defined by
$$\phi^*(y) =
  \sum_{i=1}^3( \Lambda(a_i) + a_iy_i)$$
  where $a_1(y), a_2(y), a_3(y)$ are inner angles of the generalized Euclidean
  triangle of edge lengths  $e^{y_1},$ $e^{y_2}$ and $e^{y_3}$ so that
  $a_i$ is opposite to the edge of length $e^{y_i}.$
Furthermore,

 (1) $\frac{\partial \phi^*(x)}{\partial y_i}=a_i(y),$ and

 (2) if $\theta=(\theta_1, \theta_2, \theta_3) \in \R_{>0}^3$ so
that $\theta_1+\theta_2+\theta_3 =\pi,$ then the convex function
$\psi_{\theta}(y)=\phi^*(y)-\sum_{i=1}^3 \theta_iy_i$ satisfies
that $\psi_{\theta}(y+(k,k,k))=\psi_{\theta}(y)$ for all $k\in\mathbb R$ and
$\lim_{\max{|y_i-y_j|} \to \infty} \psi_{\theta}(y) =\infty.$
\end{proposition}

\begin{proof}
For $y \in \R^3,$ define $g(x) = g_y(x) :=\sum_{i=1}^3 (x_i y_i
+\Lambda(x_i)): \overline{\mathcal A} \to \R.$ By definition,
$\phi^*(y)=\sup\{x \cdot  y -\phi(x) |x \in \R^3\}
=\sup\{\sum_{i=1}^3 x_i y_i +\Lambda(x_i)| x \in
\overline{\mathcal A}\} = \max\{ g(x) | x \in \overline{\mathcal
A}\}.$ Let $\Omega =\{(z_1, z_2, z_3)| e^{z_i} + e^{z_j}
> e^{z_k}, \text{$i,j,k$ distinct} \}.$ If $y \in \Omega,$ then
$g$ has a critical point at $(a_1(y), a_2(y), a_3(y))$ in
$\mathcal A.$ Indeed, for the tangent vector
$\frac{\partial}{\partial x_i} -\frac{\partial}{\partial x_j}$ to
$\mathcal A,$
\begin{equation}\label{210}
\begin{split}
(\frac{\partial}{\partial x_i}
-\frac{\partial}{\partial x_j})(g) =& y_i-y_j -\ln(2 |\sin(x_i)|) +
\ln(2 |\sin(x_j)|) \\=&\ln (\frac{e^{y_i}}{\sin(x_i)}) - \ln
(\frac{e^{y_j}}{\sin(x_j)}) =0
\end{split}
\end{equation}
at $x=a(y)$ due to the
sine law. Since $g$ is a strictly concave function in $x \in
\mathcal A,$ it follows that $a=(a_1, a_2, a_3)$ is the unique
maximum point of $g$ in $\mathcal A,$ i.e.,
$$\phi^*(y)=\max\{g(x) | x \in \overline{\mathcal A}\}=\sum_{i=1}^3 (a_i y_i +\Lambda(a_i)).$$ 
 If $y \notin \Delta,$ say $e^{y_i} \geqslant e^{y_j}
+ e^{y_k},$ then $a_i=\pi,$ $a_j=a_k=0.$  The same calculation
above shows that $g(x)$ has no critical points in $\mathcal A.$
Therefore $\max\{ g(x) | x \in \overline{\mathcal A}\} =\max\{g(x)
| x \in
\partial \overline{\mathcal A}\}.$ On the other hand
$(\Lambda(x_1)+\Lambda(x_2) + \Lambda(x_3))|_{\partial
\overline{\mathcal A}} =0,$ we obtain $\max\{g(x) | x \in \partial
\overline{\mathcal A}\} =\max\{ \sum_{r=1}^3 x_ry_r | x_r \geqslant 0,
x_1+x_2+x_3 =\pi, \text{and some } x_s =0 \}.$ But for $x \in
\partial \overline{\mathcal A},$ $\sum_{i=1}^3 x_r y_r \leqslant (\sum_{r=1}^3x_r) y_i =\pi
y_i$ so that equality holds for $x$ with $x_i=\pi$ and
$x_j=x_k=0.$ Therefore, $\phi^*(y) =\pi y_i=\sum_{j=1}^3
\Lambda(a_j) +a_jy_j$ due to $a_i=\pi,$ $a_j=a_k=0.$  In
particular, we see that $\phi^*$ is  convex and hence continuous
on $\R^3.$

To show $C^1$-smoothness of $\phi^*,$ we first note that by the
sine law for Euclidean triangles of angles $a_1, a_2, a_3$ and
lengths $e^{y_1}, e^{y_2}, e^{y_3},$ $d(\sum_{i=1}^3
\Lambda(a_i))=-\sum_{i=1}^3 \ln|2\sin(a_i)|da_i =-\sum_{i=1}^3
y_ida_i.$ See for instance \cite{Riv2}. Therefore, if $y \in
\Omega,$ $d\phi^*=\sum_{i=1}^3 (a_i dy_i +y_ida_i) +d(\sum_{i=1}^3
\Lambda(a_i))=\sum_{i=1}^3 a_i dy_i,$ i.e., $\bigtriangledown
\phi^*=(a_1, a_2, a_3).$ In the open set $\mathcal U_i =\{y|
e^{y_i}> e^{y_j}+e^{y_k}\}$ we have $\phi^*(y)=\pi y_i.$ Hence we
also have $\bigtriangledown \phi^*(y)=(a_1, a_2, a_3)$ on
$\mathcal U_i.$ On the other hand,  the $C^1$-smooth function
$F(y)$ defined in lemma \ref{lemma2.2} satisfies $\bigtriangledown
F=(a_1, a_2, a_3)$ on $\R^3.$ Therefore these two functions
$\phi^*$ and $F$ have the same gradient on the open dense subset
$\Omega \cup \cup_{i=1}^3 \mathcal U_i.$ Since $F$ and $\phi^*$ are continuous, these two functions defer by a
constant. Hence $\phi^*$ is $C^1$-smooth.

To prove the last statement, note that
$\psi_{\theta}(y+(k,k,k))=\psi_{\theta}(y)$ and convexity of
$\psi_{\theta}$ follow from the definition. Furthermore by
 lemma \ref{lemma2.2},
 $\psi_{\theta}$ is
strictly convex when restricted to $\{y \in \R^3| \sum_{i=1}^3
y_i=0, e^{y_i}+e^{y_j}> e^{y_k}\}.$ Let $l=(l_1, l_2, l_3) \in
\R^3$ be the vector so that the Euclidean triangle of edge lengths
$e^{l_1}, e^{l_2}, e^{l_3}$ has inner angles $\theta_1, \theta_2,
\theta_3$ and $\sum_{i=1}^3l_i=0.$ Then $\bigtriangledown
\psi_{\theta}(l)=0.$ Consider the plane $P=\{y \in \R^3
|\sum_{i=1}^3 y_i =0\}.$ The restriction $\psi_{\theta}|_P$ is a
convex function with a minimal point $l$ so that $\psi_{\theta}$
is strictly convex near $l.$ Therefore $\lim_{p \in P, p \to
\infty} \psi_{\theta}(p)=\infty.$ Due to
$\psi_{\theta}(y+(k,k,k))=\psi_{\theta}(y),$ this shows, by
projecting $\R^3$ to $P,$ that $\lim_{\max\{|y_i-y_j|\}\to \infty}
\psi_{\theta}(y)=\infty.$

\end{proof}

We remark that the function $\sum_{i=1}^3[\Lambda(a_i(y))+a_i(y)
y_i]$ has appeared before in the work of Cohn-Kenyon-Propp
\cite{CKP} and Bobenko-Pinkahl-Springborn \cite{BPS}. The
proposition is similar to the work of Colin de Verdi\'ere
\cite{CV}.

\subsection{Generalized decorated metrics, angle
assignments and volume}

Suppose $(M, \T)$ is a triangulated compact pseudo 3-manifold with
the sets of edges $E =E(\T)$ and quads $\Box=\Box(\T).$

\begin{definition}[Angle assignment] An  {\rm angle assignment} on
$(M, \T)$ is a map $\alpha: \Box \to \R_{\geqslant 0}$ so that for each
tetrahedron $\sigma\in T,$ $\sum_{q\subset\sigma}\alpha(q)=\pi.$
An angle assignment is called positive if $\alpha(q)>0$ for all
$q.$ The {\em cone angle} of $\alpha$ is defined to be
$k_{\alpha}: E \to \R_{\geqslant 0}$ where $k_{\alpha}(e)=\sum_{q \sim
e} \alpha(q).$ The {\em volume}  of $\alpha,$ denoted by
$vol(\alpha),$ is defined to be $ vol(\alpha) =\sum_{q \in \Box}
\Lambda(\alpha(q)).$   Given $k:E \to \R_{\geqslant 0},$  the space  of
all angle assignments of cone angle $k$ is denoted by $\mathcal A_k^*(M,
\T)$ and the space  of all positive angle assignments of cone
angle $k$ is denoted by $\mathcal A_k(M, \T).$
\end{definition}

By definition, $$\mathcal A_k^*(M, \T)= \{ \alpha \in \R_{\geqslant 0}^{\Box}|
k_{\alpha}=k, \forall \sigma \in T, \sum_{q \subset \sigma}
\alpha(q) =\pi\},$$ and $$ \mathcal A_k(M, \T)= \{ \alpha \in \R_{>
0}^{\Box}| k_{\alpha}=k, \forall \sigma \in T, \sum_{q \subset
\sigma} \alpha(q) =\pi\}.$$

Note that if $\mathcal A_k(M, \T) \neq \emptyset,$ then the closure of
$\mathcal A_k(M, \T)$ in $\R^{\Box}$ is $\mathcal A^*_k(M, \T).$ However, it is
possible that $\mathcal A_k(M, \T) =\emptyset$ and $\mathcal A_k^*(M, \T) \neq
\emptyset.$

 The usual angle structures on closed pseudo
3-manifolds $(M, \T)$ are positive angle assignments of cone angle
$2\pi$ at each edge.

As a consequence of lemma \ref{PP2}, we have

\begin{corollary} The volume function $vol: \mathcal A_k^*(M, \T) \to \R$
is continuous, concave and is smooth strictly concave when
restricted to $\mathcal A_k(M, \T).$
\end{corollary}

\begin{definition}[Generalized decorated metric]
A  generalized decorated metric  on $(M, \T)$ is given by a map
$l: E(\T) \to \R,$ called the {\em edge length function}.
\end{definition}

For each $l \in \R^E,$ by replacing each tetrahedron $\sigma$ in
$T$ by a generalized decorated tetrahedron whose edge lengths are
given by $l,$ we define the \it dihedral angle \rm of $l$ at an
edge $e$ in a tetrahedron $\sigma >e$ to be the corresponding
dihedral angles in the generalized decorated tetrahedron whose
edge lengths are given by $l.$ Since dihedral angles are the same
at two opposite edges in a generalized decorated tetrahedron, the
dihedral angle $\alpha=\alpha_l$ of the generalized decorated
 metric $l,$ $\alpha: \Box \to [0, \infty)$ is an angle
assignment, i.e.,  $\forall \sigma \in T,$ $\sum_{q \subset
\sigma} \alpha(q) =\pi.$ The cone angle of $\alpha$ is called the
\it cone angle \rm of the metric $l.$ We denote it by $k_l,$ i.e.,
$k_l=k_{\alpha_l}.$ 

The \it volume \rm of a generalized decorated metric $l \in \R^E$
is defined to be the volume of its dihedral angle assignment,
i.e., $vol(l)=\sum_{q \in \Box} \Lambda(\alpha(q)).$  The \it
covolume \rm of $l \in \R^E,$ denoted by $cov(l),$ is defined to
be
$$cov(l)=2 vol(l)+ l \cdot k_l = 2\sum_{ q \in \Box} \Lambda(\alpha_l(q)) +\sum_{e \in E}
l(e) k_l(e)$$ where $k_l$ is the cone angle of $l$ and $u \cdot v$
is the standard inner product of $u, v \in \R^E.$ By the
definition of cone angles, we have
\begin{equation}\label{cov}
cov(l) = \sum_{\sigma \in T}\sum_{ q\subset \sigma}[2
\Lambda(\alpha_l(q)) +\alpha_l(q)\sum_{e \sim q} l(e)].
\end{equation}


\begin{proposition}\label{prop2.10} The covolume function defined on the space of all
generalized decorated tetrahedra $\R^6$ is given by
\begin{equation}\label{345} cov(x_1, ..., x_6) =
 2 \phi^*(\frac{x_{1}+x_{4}}{2}, \frac{x_{2}+x_{5}}{2},
\frac{x_{3}+x_{6}}{2})\end{equation} where $x_1, ..., x_6$ are the
edge lengths so that $x_i$ and $x_{i+3}$ are lengths of opposite
edges. In particular, $cov: \R^6 \to \R$ is a $C^1$-smooth convex
function so that
\begin{equation}\label{schalfi}\frac{\partial cov(x)}{\partial
x_i} =\alpha_i \end{equation} when $\alpha_i$ is the dihedral
angle at the i-th edge.

\end{proposition}

\begin{proof} For $i=1,2,3,$ let $y_i=\frac{ x_i+x_{i+3}}{2}$ and
$\alpha_1, \alpha_2, \alpha_3$ be the angles of the generalized
Euclidean triangle of edge lengths $e^{y_1}, e^{y_2}, e^{y_3}.$
Then by proposition \ref{prop2.6},
$\phi^*=\sum_{i=1}^3[\Lambda(\alpha_i)+\frac{1}{2}\alpha_i(x_i+x_{i+3})]$
and $\frac{\partial \phi^*}{\partial y_i}=\alpha_i.$ Now by
definition, the dihedral angles of generalized tetrahedron of
lengths $x_i$'s are $\alpha_i$ and $\alpha_{i+3}=\alpha_i$ for
$i=1,2,3.$ Therefore by (\ref{cov}), $cov(x)=2
\sum_{i=1}^3[\Lambda(\alpha_i) +\frac{1}{2}\alpha_i(x_i+x_{i+3})]
=2\phi^*(y).$ Furthermore, for $i=1,2,3,$ $\frac{\partial
cov(x)}{\partial x_i} =2\cdot \frac{1}{2} \cdot \frac{\partial
\phi^*}{\partial y_i} =\alpha_i.$ The result also works for
$i=4,5,6$ due to $\alpha_{i+3}=\alpha_i.$


\end{proof}

\section{Volume maximization and covolume minimization}

A \it decorated ideal hyperbolic polyhedral metric \rm on $(M,
\T)$ is a generalized metric $l \in \R^E$ so that each tetrahedron
$\sigma \in T$ becomes a decorated ideal hyperbolic tetrahedron in
the length $l.$  These metrics are the same as complete finite
volume hyperbolic cone metrics on $M - V(\T)$ which are obtained
as isometric gluing of ideal tetrahedra along codimension-1 faces
together with a horosphere centered at each cusp. The collection
of the horospheres is called a \it decoration\rm.

We will establish the main results for decorated ideal hyperbolic polyhedral
metrics in this section.  These results imply theorems stated in
\S1.

\begin{theorem}\label{thm31} Suppose $(M, \T)$ is a compact
triangulated pseudo 3-manifold.
\begin{enumerate}
\item  If $l\in \R^{E(\T)}$ is a generalized decorated metric on
$(M, \T)$ with dihedral angle $\alpha=\alpha_l \in \mathcal A^*_k(M, \T)$
of cone angle $k,$ then $\alpha$ is a maximum volume point for
$vol$ on $\mathcal A^*_k(M, \T).$  \item If $k \in \R^{E(\T)}$ so that
$\mathcal A_k(M, \T) \neq \emptyset$ and $\alpha \in \mathcal A^*_k(M, \T)$ is a
maximum volume angle assignment, then there exists a generalized
decorated metric $l \in \R^{E(\T)}$ so that its dihedral angle
function is $\alpha.$ Furthermore, the maximum volume angle point
$\alpha$ is unique.
\end{enumerate}
\end{theorem}

\begin{theorem}\label{thm32} A decorated ideal hyperbolic polyhedral metric on $(M,
\T)$ is determined up to isometry and change of decoration by its
cone angles at edges. \end{theorem}

Theorems \ref{thm31} and \ref{thm32} are consequences of a duality
result for volume and covolume to be proved below.

\subsection{A combinatorics of triangulations}

For a triangulation $\T$ of edges $E=E(\T)$ and vertices
$V=V(\T),$ the vector space $\R^{V(\T)}$ acts linearly on
$\R^{E(\T)}$ by
\begin{equation}\label{action}(w+x)(vv')=w(u)+w(v')+x(vv')\end{equation}
where $ w \in \R^{V(\T)},$ $x\in \R^{E(\T)}$ and the edge $vv'$
has vertices $v, v'.$ We will identify $\R^{V(\T)}$ with the
linear subspace $\R^{V(\T)}+0$ of $\R^{E(\T)}.$ For each
tetrahedron $\sigma \in T,$ let $E(\sigma)$ and $V(\sigma)$ be the
sets of edges and vertices in $\sigma$ so that $\R^{V(\sigma)}$
acts on $\R^{E(\sigma)}$ according to (\ref{action}).  Let
$L_{\sigma}: \R^{E(\T)} \to \R^{E(\sigma)}/\R^{V(\sigma)}$ be the
composition of the restriction map and quotient map. The following
lemma was proved in \cite{neumann}, also see \cite{choi} page
1354.

\begin{lemma}[Neumann]\label{choi} The kernel of the linear map $\prod_{\sigma \in T} L_{\sigma}: \R^{E(\T)} \to \prod_{\sigma \in T} \R^{E(\sigma)}/\R^{V(\sigma)}$ is $\R^{V(\T)}.$
In particular, the induced linear map $\R^{E(\T)}/\R^{V(\T)} \to
\prod_{\sigma \in \T} \R^{E(\sigma)}/\R^{V(\sigma)}$ is injective.
\end{lemma}

Indeed, if $x \in R^{E(\T)}$ is in the kernel, then for each
tetrahedron $\sigma,$ we can find a function $f_{\sigma}:
V(\sigma) \to \R$ so that \begin{equation}\label{triangeq}
x(vv')=f_{\sigma}(v)+f_{\sigma}(v').\end{equation} The goal is to
show that if $\sigma$ and $\sigma'$ are two tetrahedra sharing the
same vertex $v,$ then $f_{\sigma}(v)=f_{\sigma'}(v).$ Now if
$\sigma$ and $\sigma'$ have a common triangle face $t>v,$ then the
result follows by considering the three equations (\ref{triangeq})
at three vertices of $t.$ The geneneral case follows by producing
a sequence of tetrahedra $\sigma_0=\sigma,$ $\sigma_1,...,
\sigma_m=\sigma'$ so that $\sigma_i$ and $\sigma_{i+1}$ have a
common triangular face $t_i>v.$

\subsection{The covolume functions}

For a triangulated pseudo 3-manifold $(M, \T),$  define
\begin{equation}\label{cone}
\mathcal C(\T) =\{ k \in \R^{E(\T)}| \mathcal A^*_k(M, \T) \neq \emptyset\}
\end{equation}
to be the space of all cone angles of angle assignments. For a tetrahedron $\sigma$ with sets of edges $E(\sigma)$ and
vertices $V(\sigma),$ the covolume function
$cov_{\sigma}:\R^{E(\sigma)}\to \R$ sends $x \in \R^{E( \sigma)}$
to $ \sum_{q \subset \sigma}[2 \Lambda(\alpha(q))
+\alpha(q)\sum_{e \sim q} x(e)]$ is $C^1$-smooth and convex. By
proposition \ref{prop2.10}, we have $\frac{\partial
cov_{\sigma}(x)}{\partial x(e)}=\sum_{q \sim e} \alpha(q).$
Furthermore, by lemma \ref{lemma2.2}, for any $k \in \mathcal
C(\sigma)$ and $w \in \R^{V(\sigma)},$ the function
$cov_{\sigma}(x)-x \cdot k$ satisfies
\begin{equation}\label{11234}cov_{\sigma}(w+x)-(w+x)\cdot k
=cov_{\sigma}(x)-x \cdot k. \end{equation} In particular,
$cov_{\sigma}$ is a convex function defined on the quotient space
$\R^{E(\sigma)}/\R^{V(\sigma)}.$ Furthermore, if $k=k_{\theta}$ so
that $\theta(q) >0$ for all $q \in \Box$, then

\begin{equation}\label{limit1}
\lim_{ x \to \infty,\ x \in \R^{E(\sigma)}/\R^{V(\sigma)}}
[cov_{\sigma}(x)-x\cdot k ]=+\infty.
\end{equation}

This follows from corollary 2.10 and lemma 2.2.
For
$k_{\theta}\in \mathcal C(\T)$ (i.e., $k_{\theta}(e)=\sum_{q \sim
e} \theta(q)$ for an angle assignment $\theta$), consider the
function $cov_{\T}(x)-x \cdot k_{\theta}.$

\begin{proposition}\label{prop3.4}

(1) For any $w \in \R^{V(\T)},$ $cov_{\T}(w+x)-(w+x)\cdot
k_{\theta} =con_{\T}(x)-x \cdot k_{\theta}.$

(2) If $\theta(q)>0$ for all $q,$ then
$$\lim_{x \to \infty,\ x \in \R^{E(\T)}/\R^{V(\T)}}  [cov_{\T}(x)-x\cdot k] =+\infty.$$
In particular, $cov_{\T}-x \cdot k_{\theta}$ has a minimal point
in $\R^{E(\T)}.$
\end{proposition}

\begin{proof} We can rewrite the function as
$$cov_{\T}(x)-x \cdot k_{\theta}=\sum_{\sigma \in T}
[ \sum_{ q \subset \sigma}
2\Lambda(\alpha(q))+(\alpha(q)-\theta(q))\sum_{e \sim q} x(e)].$$
Therefore, statement (1) follows from that of covolume function
for single tetrahedron (\ref{11234}). Part (2) follows from the
condition that $\theta(q)>0$ for all $q,$ lemma \ref{choi} and
(\ref{limit1}). Since the function $cov_{\T}(x)-x\cdot k$ is
convex on $\R^E$ and is invariant under the linear action of
$\R^V,$ it is a convex function on the Euclidean space $\R^E/\R^V$
which tends to infinity as $x$ tends to infinity. Therefore,
$cov_{\T}(x)-x\cdot k$ has a minimal point in $\R^E/\R^V$ which
implies that it has a minimal point in $\R^E.$

\end{proof}

\subsection{Fenchel dual of volume}

The space of all cone angles $\mathcal C (\T)$ defined by
(\ref{cone}) is a compact convex polytope in $\R^E.$ Indeed, it is
the image of the compact convex polytope of all angle assignments
in $\R^{\Box}$ under a linear map.  Define
$W:\mathbb{R}^E\rightarrow \mathbb{R}$ to be the function
\begin{equation}\label{vol1}
W(k)=\left\{
\begin{array}{cl} 2 \min\{-vol(a)\ |\ a\in \mathcal A_k^*(M,\T)\} & \text{ if }k\in\mathcal C(\T),\\
+\infty & \text{ if }k\notin \mathcal C(\T).
\end{array} \right.
\end{equation}
The function $W(k)$ encodes the volume optimization program. For
instance, $W(2\pi, ..., 2\pi)$ is the Casson-Rivin's program of
finding complete hyperbolic metrics of finite volume.

\begin{proposition}\label{continuous} The function
$W: \mathcal C(\T) \to \R$ is convex and lower semi-continuous
 in the compact convex set $\mathcal C(\T).$
\end{proposition}

This proposition follows from the  lemma \ref{4561} below by
taking $X =\{ x \in \R_{\geqslant 0}^{\Box} |$$ \forall \sigma \in T,
\sum_{q \subset \sigma}$$ x(q)=\pi\},$ $f(x) =-2\sum_{q \in \Box}
\Lambda(x(q))=-2\sum_{\sigma \in T}\sum_{q \subset \sigma}
\Lambda(x(q))$ and $L:\R^{\Box} \to \R^E$ to be $L(x)(e) =\sum_{q
\sim e} x(q).$

\begin{lemma}\label{4561} Suppose $X \subset \R^m$ is a compact convex set,
$f: X \to \R$ is a continuous convex function and $L: \R^n \to
\R^m$ is linear. Then $g(y) =\min\{ f(x) | x \in X, L(x) =y\}$ is
convex and lower semi-continuous  on $L(X).$
\end{lemma}

\begin{proof}
Take  $y_1, y_2 \in L(X)$ and choose $x_1, x_2 \in X$ so that
$L(x_i)=y_i$ and $g(y_i) = f(x_i)$ for $i=1,2$ and $t \in [0,1].$
Then $L(tx_1+(1-t)x_2) = ty_1+(1-t)y_2.$ Therefore, $g(ty_1+(1-t)
y_2) \leqslant f(tx_1+(1-t)x_2) \leqslant tf(x_1)+(1-t)f(x_2)
=tg(y_1)+(1-t)g(y_2),$ i.e., $g$ is convex. To see lower semi
continuity of $g,$ suppose $y_n \in L(X)$ so that $\lim_n y_n = b
\in L(X).$ By the compactness of $X,$ after selecting a
subsequence, we may assume that $y_n = L(x_n)$ so that $\lim_{n}
x_n =a$ in $X$ and $g(y_n) =f(x_n).$ Clearly due to continuity
$\lim_n f(x_n) =f(a)$ and $L(a)=b.$ This shows that $\lim_n g(y_n)
= \lim_n f(x_n) =f(a) \geqslant g(b).$
\end{proof}

One of the two main technical results of the paper is the following,

\begin{theorem}\label{dual}  The Fenchel dual of the $C^1$ smooth convex
covolume $cov: \R^E \to \R$ is the lower semi-continuous convex
function function $W: \R^E \to (-\infty, \infty]$ defined by
(\ref{vol1}) .

\end{theorem}

 \begin{proof} By the Fenchel duality theorem for convex function,
  it suffices to show the
 dual $cov^*(y) =\sup\{x\cdot y-cov(x) | x \in \R^E\}$ of $cov$ is the function $W.$
 Take $y \in \R^E.$ We will show that if $y \notin \mathcal C(\T),$ then
 $cov^*(y) =\infty$ and if $y \in \mathcal C(\T),$
 then $cov^*(y) = W(y).$

 {\bf Case 1.} $y \notin \mathcal C(\T).$  Suppose otherwise that $cov^*(y) <
 \infty,$ i.e., there exists $C_1>0$ so that $x \cdot y-cov(x) \leqslant C_1$
 for all $x \in \R^E.$ Since $vol(x)$ is uniformly bounded
 for all $x$ and $cov(x)-x \cdot k_x = 2vol(x),$ there exists $C_2>0$ so that for all $x \in \R^E,$
 \begin{equation}\label{eqn:221}
 x \cdot y -x  \cdot k_x \leqslant C_2.
 \end{equation} where $k_x$ is the cone angle of $x$ considered as a generalized decorated metric.

Consider the linear map $L: \R^{\Box} \to \R^E \times \R^T$
defined by
$$ L(z)(e) =\sum_{ q \sim e} z(q), \quad L(z)(\sigma)
=\sum_{ q \subset \sigma} z(q).$$ By the assumption that $y \notin
\mathcal C(\T),$  $(y, \pi, ..., \pi) \notin L(\R^{\Box}_{\geqslant
0}).$ Since $L(\R^{\Box}_{\geqslant 0})$ is a closed convex cone in
$\R^E \times \R^T,$ by the separation theorem applied to $(y, \pi,
..., \pi)$ and $L(\R^{\Box}_{\geqslant 0}),$ there exists $h \in \R^E
\times \R^T$ so that

$$  <h, (y, \pi, ..., \pi)> = C_2+1,$$ and for all $t \in \R^{\Box}_{\geqslant
0}$
$$ < h, L(t)> \leqslant 0$$
where $<u,v>$ is the standard inner product in $\R^E \times \R^T.$

We can rewrite above two inequalities as
\begin{equation} \label{eqn:31}
\sum_{e \in E} h(e)y(e)  + \pi \sum_{\sigma \in T} h(\sigma)
=C_2+1,
\end{equation} and,
\begin{equation}\label{eqn:44}
\sum_{e \in E} [h(e) \sum_{q \sim e} t(q) ] +\sum_{\sigma \in T}[
h(\sigma) \sum_{q \subset \sigma} t(q) ] \leqslant 0.
\end{equation}

Equation (\ref{eqn:44}) can be written as,
$$\sum_{q \in \Box} t(q) (\sum_{e \sim q} h(e) + \sum_{q \subset \sigma}
h(\sigma)) \leqslant 0.$$ Since it holds  for all $t \in
\R^{\Box}_{\geqslant 0},$ hence for all $q \in \Box$
\begin{equation}\label{eqn:51}
\sum_{e \sim q} h(e) + \sum_{q \subset \sigma} h(\sigma) \leqslant
0.
\end{equation}

Now since (\ref{eqn:221})  holds for all $x \in \R^E,$ it holds
for the  projection of $h \in \R^E \times \R^T$ to $\R^E.$ Taking
this projection as $x$ in (\ref{eqn:221}), letting the diehedral
angle of $x$ be $\alpha=\alpha_x$ and using (\ref{eqn:31}) and
(\ref{eqn:44}), we obtain
\begin{equation*}
\begin{split}
C_2 \geqslant& \sum_{e\in E} h(e) y(e) -
\sum_{e\in E} h(e)(\sum_{q \sim e} \alpha(q))\\
=&\,C_2+1 -\pi \sum_{\sigma \in T} h(\sigma) -
\sum_{e\in E} (\sum_{ q \sim e} \alpha(q)) h(e)\\
=&\,C_2+1 -  \sum_{\sigma \in T} h(\sigma)(\sum_{q \subset \sigma}
\alpha(q))     -\sum_{e\in E}(\sum_{q \sim e} \alpha(q)) h(e)\\
=&\,C_2+1 -\sum_{q \in \Box} \alpha(q)( \sum_{q \subset \sigma}
h(\sigma)+         \sum_{e \sim q} h(e) )\\
\geqslant&\,C_2+1.\\
\end{split}
\end{equation*}
Here we have used (\ref{eqn:31}) and (\ref{eqn:51}) in steps 2 and
4 in the derivation above.  This is a contradiction. Therefore,
$cov^*(y) =\infty.$

{\bf Case 2.} $y \in \mathcal C(\T),$ say $y =k_{\theta}$ for an angle
assignment $\theta  \in \R_{\geqslant 0}^{\Box}$ where $y(e) =\sum_{q
\sim e} \theta(q).$  For this choice of $\theta,$ and for any $x
\in \R^E$ with dihedral angles $\alpha(q) =\alpha_x(q),$ we can
write
$$ x \cdot y -cov(x)= \sum_e x(e)y(e) -2vol(x) -x \cdot k_x$$
$$=-\sum_{\sigma \in T} [\sum_{q \subset \sigma}( 2\Lambda(\alpha(q))
+\sum_{e < \sigma} x(e)\sum_{q \sim e} (\alpha(q) -\theta(q))].$$

\begin{lemma}\label{3227}
If $\theta_1, \theta_2, \theta_3 \geqslant 0$ and
$\theta_1+\theta_2+\theta_3=\pi,$ then for any $x_1, x_2, x_3 \in
\R$ so that $\alpha_1, \alpha_2, \alpha_3$ are inner angles of the
generalized Euclidean triangle of edge lengths $e^{x_1}, e^{x_2},
e^{x_3},$ we have
$$\sum_{ i=1}^3 \Lambda(\alpha_i) +\sum_{i=1}^3
x_i(\alpha_i -\theta_i) \geqslant \sum_{i=1}^3 \Lambda(\theta_i).$$
\end{lemma}
\begin{proof}
The lemma follows from proposition \ref{prop2.6} on  convexity and
$C^1$-smoothness of the covolume function $cov: \R^3 \to \R$ (for
single tetrahedron). For any $C^1$-smooth convex function $F$ on
$\R^n,$ we have $F(x) -DF(l)(x-l) \geqslant F(l).$ Now using the fact
that $\frac{\partial cov}{\partial x_i} = \alpha_i$ and take $l\in
\R^3$ so that the angles of the generalized Euclidean triangle of
edge lengths $e^{l_1}, e^{l_2}, e^{l_3}$ are $\theta_i$'s, the
result follows.
\end{proof}

Summing up the inequalities in lemma \ref{3227} over all
tetrahedra in $T,$ we obtain
$$ 2 vol(x) + x \cdot k_x -x \cdot k_{\theta} \geqslant 2 vol(\theta).$$  This
implies, for any angle assignment $\theta \in \mathcal A^*_y(M, \T) ,$ we
have
$$ x \cdot y -cov(x) \leqslant -2vol(\theta),$$ and hence \begin{equation}
\label{432} x  \cdot y -cov(x) \leqslant  W(y). \end{equation}

Taking the supremum in $x$  of the inequality above, we obtain
$$cov^*(y) \leqslant W(y).$$

We claim that the equality holds. There are two steps involved in
the proof.  In the first step, we assume that $y=k_{\theta}$ where
$\theta \in \mathcal A_y(M, \T).$ In this case, by proposition
\ref{prop3.4},
the function $cov(x) -x \cdot y$ of variable $x$ has a minimal
point $x^*$ in $\R^E.$  Let $\theta^*$ be the dihedral angles of
$x^*.$ Then by $\frac{\partial cov(x)}{\partial x_e} = k_x(e),$ we
obtain  $k_{x^*}=y$ (due to $\frac{\partial (cov(x)- x \cdot
y)}{\partial x_e} = k_x(e)-y(e) =0$ for $x=x^*$). Therefore, at
this point $x^*,$ $x^* \cdot y -cov(x^*) = -2vol(\theta^*) \geqslant
W(y).$ This shows $cov^*(y) \geqslant W(y).$ Combining with
(\ref{432}), we conclude $cov^*(y)=W(y).$

Next, both $cov^*(y)$ and $W(y)$ are convex and semi-continuous
 on the closed convex set $\mathcal C(\T)$ so that they
coincide in the subset $\{ y  | \mathcal A_y(M, \T) \neq \emptyset \}$ of
$\mathcal C(\T)$ which contains the relative interior of $\mathcal
C(\T).$  Therefore, $cov^*(y)=W(y)$ on $\mathcal C(\T)$ is a
consequence of the lemma below.
\end{proof}

\begin{lemma}[\cite{roc}, corollary 7.3.4]\label{relativeint} Suppose $X \subset \R^n$ is a closed convex set and
$f,g: X \to \R$ are convex semi-continuous functions. If $f$ and
$g$ coincide on the relative interior of $X,$ then $f=g.$
\end{lemma}

\subsection{Proofs of main theorems}

We begin with the following,

\begin{proposition}\label{uniq123} Suppose $k\in \R^E$ so that $\mathcal A_k(M, \T) \neq
\emptyset.$ Then there is a unique maximum point of the volume
function $vol: \mathcal A^*_k(M, \T) \to \R.$
\end{proposition}
\begin{proof} The maximum point exists since the $vol$ is
continuous on the compact set $\mathcal A^*_k(M, \T).$ Suppose otherwise
that $\alpha, \alpha'$ are two distinct maximum points of volume
in $\mathcal A^*_k(M, \T).$ Then due to convexity, all points in the line
segment $t\alpha+(1-t)\alpha'$ for $t \in [0, 1]$ are maximum
point. On the other hand, the work of Rivin \cite{Riv2}
shows that,  due to $\mathcal A_k(M, \T) \neq \emptyset,$  for any maximum
volume point $\beta \in \mathcal A_k^*(M, \T)$ if $\beta(q)=0$ for some $q
\subset \sigma,$ then for the other two quads $q', q'' \subset
\sigma,$ $\{\beta(q'), \beta(q'')\}=\{0, \pi\}.$ This shows if
$\alpha(q)=\pi,$ then $\alpha'(q)=\pi.$ For otherwise,  the values
of the maximum point $\frac{1}2( \alpha+\alpha')$ at the three
quads $q, q', q'' \subset \sigma$ would be $0, \mu, \pi-\mu$ for
some $\mu \in (0, \pi).$ Now  due to $\mathcal A_k(M, \T) \neq \emptyset,$
$vol(\alpha)>0.$ Therefore, there is a tetrahedron $\sigma \in T$
so that $\alpha(q), \alpha'(q) >0$ for all $q \subset \sigma.$
Hence the function $g(t) =vol(t\alpha+(1-t)\alpha')=\sum_{\sigma
\in T} \sum_{ q \subset \sigma} \Lambda(t \alpha(q)
+(1-t)\alpha'(q))$ is a sum of concave functions so that one of it
is strictly concave in $t.$ This implies $g(t)$ is not a constant
function which contradicts that $g(t)$'s are the maximum value
$vol(\alpha).$
\end{proof}

\subsubsection{A proof of theorem \ref{thm31}}

To prove part (1) of theorem \ref{thm31}, take any $\theta \in
\mathcal A^*_k(M, \T)$ whose cone angle is $k.$ By definition $W(k)\leqslant
-2vol(\theta).$ But by the duality theorem \ref{dual}, $W(k)
=cov^*(k) =\sup\{ y \cdot k -cov(y) | y \in \R^E\} \geqslant x \cdot k
-cov(x) =-2vol(\alpha).$ Thus $vol(\alpha) \geqslant vol(\theta),$
i.e., $\alpha$ is a maximum volume point.

To prove part (2) of theorem \ref{thm31}, since $\mathcal A_k(M, \T) \neq
\emptyset,$ by proposition \ref{prop3.4},  the function $cov(x)-x
\cdot k$ has a minimal point $l \in \R^E.$ Let $\beta $ be the
dihedral angle of $l.$ Then by $\bigtriangledown (cov(x)-x \cdot
k)=0$ at $x=l,$ we see that the cone angle of $\beta$ is $k,$
i.e., $\beta \in \mathcal A^*_k(M, \T).$ By part (1), $\beta$ is a maximum
volume point in $\mathcal A^*_k(M, \T).$ Therefore, both $\beta$ and
$\alpha$ are maximum volume points in $\mathcal A^*_k(M, \T).$ By
proposition \ref{uniq123}, $\alpha=\beta.$ Thus the result
follows.

\subsubsection{A proof of theorem \ref{thm32}}

To prove  theorem \ref{thm32}, suppose that $x,y \in \R^E$ are two
decorated ideal hyperbolic polyhedral metrics on $(M, \T)$ so that
their cone angles $k_x$ and $k_y$ are the same. We will show that
$y=w+x$ for some $w \in \R^{V(\T)},$ i.e., $x,y$ differ by a
change of decoration. Let $k=k_x.$ By the assumption that $x$ is a
decorated ideal hyperbolic polyhedral metric, $\mathcal A_k(M,\T)
\neq \emptyset.$ 
By theorem \ref{thm31}, the
dihedral angles $\alpha_x$ and $\alpha_y$ of $x$ and $y$ are the
maximum volume points in $\mathcal A_k^*(M, \T).$ On the other
hand, proposition \ref{uniq123} shows the maximum volume point is
unique. Therefore $\alpha_x=\alpha_y.$ This implies that the
underlying hyperbolic metrics for $x$ and $y$ are isometric.
Therefore, $x$ and $y$ differ by a change of decoration, i.e.,
$x=w+y,$ for some $w \in \R^V.$

Note that we have proved a slightly stronger statement that we
only need to assume $x$ is decorated hyperbolic metric and $y$ is
a generalized decorated metric of the same cone angle.

\section{Hyper-ideal tetrahedra}\label{Hyperideal}

\subsection{Preliminaries}\label{prelim}

We recall some of the basic results on hyper-ideal tetrahedra in
this subsection. Following \cite{BB} and \cite{Fu}, a
\emph{hyper-ideal tetrahedron} $\sigma$ in $\mathbb{H}^3$ is a
compact convex polyhedron that is diffeomorphic to a truncated
tetrahedron in $\mathbb{E}^3$ with four hexagonal faces
right-angled hyperbolic
 hexagons (see Figure \ref{hyper-ideal} (a)). The four triangular faces isometric to hyperbolic triangles
 are called \emph{vertex triangles}. An \emph{edge} in a hyper-ideal tetrahedron is the
 intersection of two hexagonal faces, and a \emph{vertex edge} is the intersection of a hexagonal
  face and a vertex triangle. The \emph{dihedral angle} at an edge is the angle between the two hexagonal faces
   adjacent to it. The dihedral angle between a hexagonal face and
   a vertex triangle is always $\pi/2.$

Let $\Delta_i,$ $i=1,2,3,4,$ be the four vertex triangles of
$\sigma.$ We use $e_{ij}$ to denote the edge joining $\Delta_i$ to
$\Delta_j,$ and use $H_{ijk}$ to denote the hexagonal face
adjacent to $e_{ij},$ $e_{jk}$ and $e_{ik}.$ (See Figure
\ref{hyper-ideal} (a).) The length of $e_{ij}$ is denoted by
$l_{ij}$ and the dihedral angle at $e_{ij}$ is denoted by $a_{ij}.$
The length of the vertex edge $\Delta_i \cap H_{ijk}$ is denoted
by $x^i_{jk}.$ As a convention, we always assume $l_{ij}=l_{ji}$ and $a_{ij}=a_{ji}.$

\begin{figure}[htbp]\centering
\includegraphics[width=12cm]{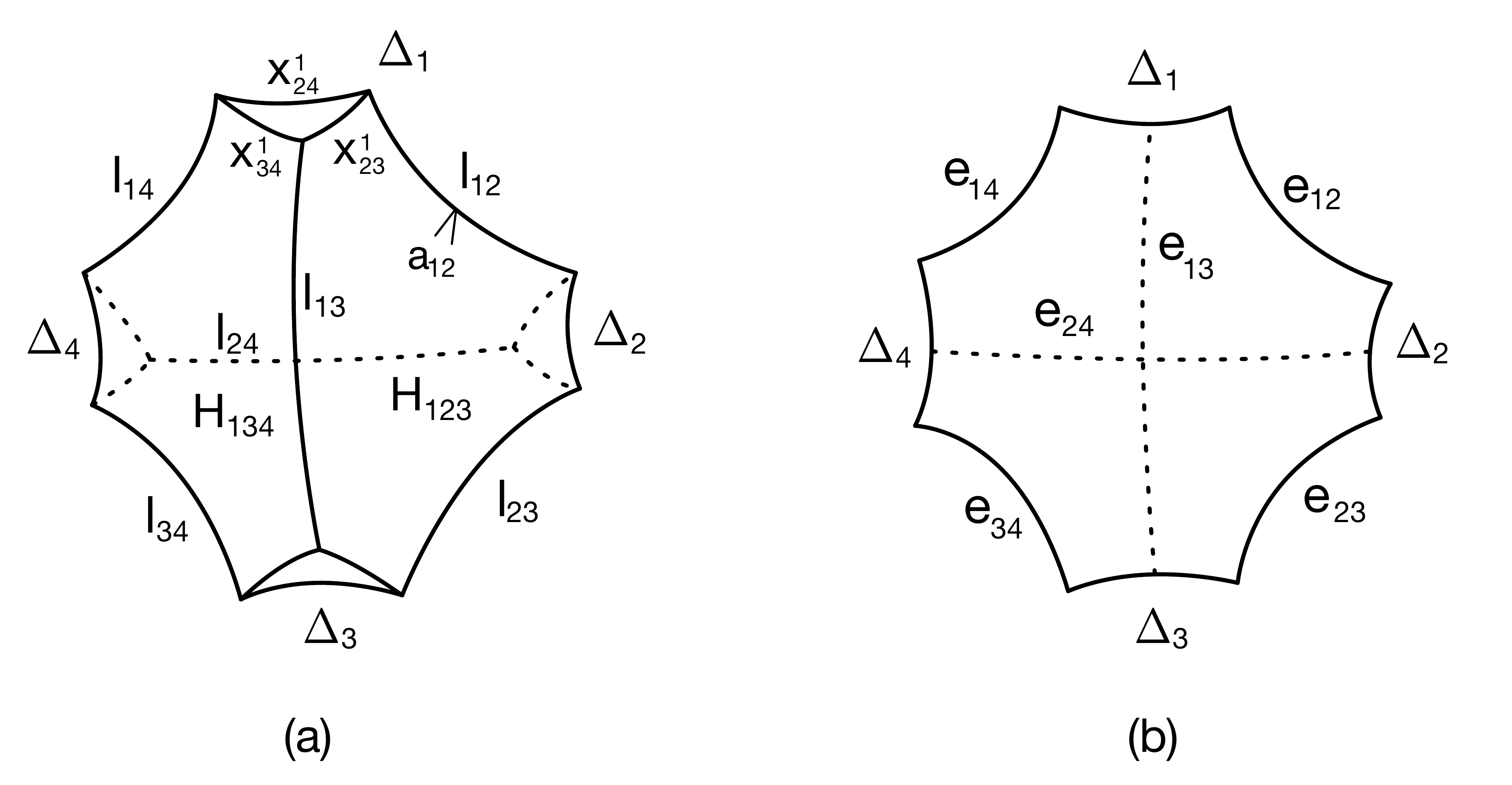}\\
\caption{Hyper-ideal and flat hyper-ideal
tetrahedra.}\label{hyper-ideal}
\end{figure}

\begin{proposition}\label{BB-Fu} (\cite{BB}, \cite{Fu}) Suppose $\sigma$ is a hyper-ideal tetrahedron in $\mathbb{H}^3.$
\begin{enumerate}[(a)]
\item The isometry class of $\sigma$ is determined by its dihedral
angle vector  $(a_{12}, \dots, a_{34}) \in \mathbb{R}^6$ which
satisfies the condition that $a_{ij}>0,$ and $\sum_{j\neq
i}a_{ij}<\pi$ for each fixed $i.$

\item Conversely, given $(a_{12},\dots,
a_{34})\in\mathbb{R}^6_{>0}$ so that $\sum_{j\neq i}a_{ij}<\pi$
for each $i,$ where $a_{ij}=a_{ji},$ there exists a hyper-ideal
tetrahedron having $a_{ij}$ as its dihedral angle at the $ij$-th
edge.

\item The isometry class of $\sigma$ is determined by its edge
length vector $(l_{12}, \dots, l_{34})\in\mathbb{R}_{>0}^6.$
\end{enumerate}
\end{proposition}

Thus, the space of isometry classes of hyper-ideal tetrahedra
parametrized by dihedral angles is the open convex polytope
\begin{equation}\label{Equation}
\mathcal B=\big\{ (a_{12}, \dots, a_{34})\in\mathbb{R}_{>0}^6\ \big|\ \sum_{j\neq i}a_{ij}<\pi \text{ for each } i,\text{ where }a_{ij}=a_{ji}\ \big\}.
\end{equation}

Let $vol\co\mathcal B\rightarrow\mathbb{R}$ be the hyperbolic
volume of the hyper-ideal tetrahedra considered as a function in
the dihedral angles. Then the Schlaefli formula says
$$\frac{\partial vol}{\partial a_{ij}}=-\frac{l_{ij}}{2}.$$
See \cite{B} for the Schlaefli formula in a more general setting.

\begin{proposition}\label{vol} The volume function has the following properties.

\begin{enumerate}[(a)]
\item (\cite{Sch}) The volume function $vol\co\mathcal
B\rightarrow\mathbb{R}$ is smooth and has positive definite
Hessian matrix at each point in $\mathcal B.$

\item (\cite{Riv}) The function $vol$ can be extended continuously
to the compact closure $\overline{\mathcal B}$ of $\mathcal B$ in
$\mathbb{R}^6,$ where
$$\overline{\mathcal B}=\big\{ (a_{12}, \dots, a_{34})\in\mathbb{R}_{\geqslant0}^6\ \big|\ \sum_{j\neq i}a_{ij}\leqslant\pi \text{ for each } i,\text{ where }a_{ij}=a_{ji}\ \big\}.$$
\end{enumerate}
\end{proposition}

Let $\mathcal L$ be the set of vectors $(l_{12},\dots,l_{34})$
such that there exists a hyper-ideal tetrahedron having $l_{ij}$
as the length of the $ij$-th edge. It follows from Proposition
\ref{BB-Fu} that $\mathcal L$ is a simply connected open subset of
$\mathbb{R}^6.$ The volume function can be considered as defined
on $\mathcal L.$ The Legendre transform of $vol\co\mathcal B
\rightarrow\mathbb{R},$ to be called the \emph{co-volume}
function, is   $cov \co\mathcal L\rightarrow\mathbb{R}$ given by
\begin{equation}\label{2.1}
cov (l)=2vol(l)+\sum_{i< j}a_{ij}l_{ij}.
\end{equation}
It is known (see \cite{Luo1}) that $cov$ has a positive definite Hessian matrix at
 each $l\in\mathcal L$ and hence $cov$ is locally strictly convex. However, the open subset $\mathcal L\subset\mathbb{R}_{>0}^6$ is
not convex. One of the main technical results in this paper is that
$cov$ can be extended to a $C^1$-smooth and convex function on
 $\mathbb{R}^6$
 by studying the
flat hyper-ideal tetrahedra.

Recall that a \emph{flat hyper-ideal tetrahedron} is defined as follows.
Take a right-angled hyperbolic octagon $Q$ with eight edges
cyclically labelled as $\Delta_1, $$e_{12},$$ \Delta_2,$$ e_{23},
$$\Delta_3, $$e_{34}, $$\Delta_4, e_{41}.$  Let $e_{13}$
(and $e_{24}$) be the shortest geodesic arc in $Q$ joining
$\Delta_1$ to $\Delta_3$  (and $\Delta_2$ and $\Delta_4$). We call
$(Q, \{e_{ij} \})$ a \it flat hyper-ideal tetrahedron \rm with six
edges $e_{ij}.$ See Figure \ref{hyper-ideal} (b). The dihedral
angles at $e_{13}$ and $e_{24}$ are $\pi$ and are $0$ at all other
edges. The volume of a flat hyper-ideal tetrahedron is defined to be zero.


\subsection{Generalized hyper-ideal tetrahedra, dihedral angles and volume}


In this subsection, we investigate the space of hyper-ideal
tetrahedra parametrized by  the edge lengths and their
degenerations. One of the goals is to extend the locally convex
function $cov$ to a convex function defined on
$\mathbb{R}^6_{>0}.$ To this end, let us define \it a generalized
hyper-ideal tetrahedron \rm to be a topological truncated
tetrahedron so that each edge is assigned a positive number,
called the edge length. We will define the \it dihedral angles \rm
and \it volume \rm
and \it covolume \rm
of generalized hyper-ideal tetrahedra in
this section.

Suppose $\sigma$ is a generalized hyper-ideal tetrahedron with
edges $e_{ij}$ joining the i-th and the j-th vertices and
$l_{ij}=l_{ji}$ is the edge length of $e_{ij}.$ To define the
dihedral angle $a_{ij}$ at $e_{ij},$ we need the following
compatibility property, which is a special case of Proposition 3.1
of \cite{Luo2}.

\begin{lemma}\label{comp} For $(l_{12},\dots,l_{34})\in\mathbb{R}_{>0}^6$ and $\{i,j,k,h\}=\{1,2,3,4\},$ let $l_{ji}=l_{ij}$ for $i\neq j$ and let
\begin{equation}\label{xijk}x^i_{jk}=\cosh^{-1}\Big(\frac{\cosh l_{ij}\cosh l_{ik}+\cosh l_{jk}}{\sinh l_{ij}\sinh l_{ik}}\Big)\end{equation}
and  \begin{equation}\label{eq5.2}\phi^i_{kh}=\frac{ \cosh
x^i_{jk}\cosh x^i_{jh}-\cosh x^i_{kh}}{ \sinh x^i_{jk} \sinh
x^i_{jh}}.\end{equation} Then $\phi^i_{kh}=\phi^j_{kh}.$
\end{lemma}


\begin{proof} Let $c_{ij}=\cosh l_{ij},$ $s_{ij}=\sinh l_{ij},$ $c^k_{ij}=\cosh x^k_{ij}$ and $s^k_{ij}=\sinh x^k_{ij}$ for $\{i,j,k\}\subset\{1,\dots,4\}.$ By definition, we have
\begin{equation}\label{43}
\begin{split}
\phi^i_{kh}=&\,\frac{1}{s^i_{jk}s^i_{jh}}\Big(\frac{c_{ij}c_{ik}+c_{jk}}{s_{ij}s_{ik}}\frac{c_{ij}c_{ih}+c_{jh}}{s_{ij}s_{ih}}-\frac{c_{ik}c_{ih}+c_{kh}}{s_{ik}s_{ih}}\Big)\\
=&\,\frac{c_{ik}c_{ih}+c_{jk}c_{jh}+c_{ij}c_{ik}c_{jh}+c_{ij}c_{ih}c_{jk}-s_{ij}^2c_{kh}}{s^i_{jk}s^i_{jh}s_{ij}^2s_{ik}s_{ih}},
\end{split}
\end{equation}
and similarly
\begin{equation}\label{44}
\phi^j_{kh}=\frac{c_{jk}c_{jh}+c_{ik}c_{ih}+c_{ij}c_{ih}c_{jk}+c_{ij}c_{ik}c_{jh}-s_{ij}^2c_{kh}}{s^j_{ik}s^j_{ih}s_{ij}^2s_{jk}s_{jh}}.
\end{equation}
To see $\phi_{kh}^i=\phi_{kh}^j,$ it suffices to show that the two denominators in (\ref{43}) and (\ref{44}) are the same. To this end, we have
\begin{equation*}
\begin{split}
(s^i_{jk}s_{ij}s_{ik})^2=&\big((c^i_{jk})^2-1\big)s_{ij}^2s_{ik}^2\\
=&\Big(\big(\frac{c_{ij}c_{ik}+c_{jk}}{s_{ij}s_{ik}}\big)^2-1\Big)s_{ij}^2s_{ik}^2\\
=&\,2c_{ij}c_{ik}c_{jk}+c_{ij}^2+c_{ik}^2+c_{jk}^2-1.
\end{split}
\end{equation*}
 Therefore,
\begin{equation}\label{eq5.4}
\begin{split}
\phi_{kh}^i(l)=\phi_{kh}^j(l)=\frac{c_{ik}c_{ih}+c_{jk}c_{jh}+c_{ij}c_{ik}c_{jh}+c_{ij}c_{ih}c_{jk}-s_{ij}^2c_{kh}}
{\sqrt{2c_{ij}c_{ik}c_{jk}+c_{ij}^2+c_{ik}^2+c_{jk}^2-1}\sqrt{2c_{ij}c_{ih}c_{jh}+c_{ij}^2+c_{ih}^2+c_{jh}^2-1}}.
\end{split}
\end{equation}

\end{proof}

Note that if $\sigma$ is a hyper-ideal tetrahedron of edge lengths
$l_{ij},$ then by the cosine law, $x^i_{jk}$ and
$\arccos(\phi_{kh})$ in lemma \ref{comp} are the lengths of the
vertex edge $\Delta_i \cap H_{ijk}$ and the dihedral angle at
$e_{kh}.$  In particular, the conclusion of the lemma is obvious for $\sigma.$
For a generalized hyper-ideal tetrahedron $(\sigma,
\{l_{ij}\}),$  using lemma \ref{comp} we call $x^i_{jk}$ the \it
length \rm of the vertex edge and define the function
$\phi_{ij}\co\mathbb R^6_{>0}\rightarrow\mathbb R$ by $\phi_{ij}(l)=\phi_{kh}^i(l).$ Then
we have

\begin{proposition}\label{lengths}  The space of all hyper-ideal tetrahedra parametrized by the edge lengths is
$$\mathcal{L}=\big\{l\in\mathbb{R}^6_{>0}\ \big|\ \phi_{ij}(l)\in(-1,1)\ \text{for all }
\{i,j\}\subset\{1,2,3,4\}, i\neq j\big\}.$$
\end{proposition}

\begin{proof} If $l$ is the edge length vector of a hyperbolic tetrahedra, then each of the dihedral angles
$a_{ij}(l)\in(0,\pi).$ By the cosine law for the hyperbolic
triangle $\Delta_k$ and right-angled hexagon $H_{ijk},$ we have
$\phi_{ij}(l)=\cos a_{ij}(l)\in(-1,1).$  This shows $\phi_{ij} \in
(-1, 1).$ Conversely, for each $l\in\mathbb{R}^6_{>0}$ with
$\phi_{ij}(l)\in(-1,1)$ for all $\{i,j\}\subset\{1,\dots,4\},$ by
(\ref{eq5.2}), $\ x^i_{jk},$ $x^i_{jh}$ and $x^i_{kh}$ satisfy the
triangular inequality. Then there exists a unique hyperbolic
triangle $\Delta_{i}$ having them as edge lengths. Taking $
a_{ij}=\cos^{-1}(\phi_{ij})\in(0,\pi),$ by Lemma \ref{comp}, we
see that $ a_{ij},$ $ a_{ik}$ and $ a_{ih}$ are the angles of
$\Delta_{i}.$ Hence they satisfy $ a_{ij}+ a_{ik}+ a_{ih}<\pi.$ By
Proposition \ref{BB-Fu}, there is a unique hyper-ideal tetrahedron
$\sigma$ with dihedral angles $\{ a_{ij}\}.$ Applying the Cosine
Law to the vertex triangles $\{\Delta_{i}\}$ and hexagons
$\{H_{ijk}\},$ we see that $l$ is the edge lengths of $\sigma.$
\end{proof}

\begin{proposition}\label{degeneration}
Let $\partial \mathcal L$ be the frontier of $\mathcal L$ in $\mathbb{R}_{>0}^6.$ Then $\partial \mathcal L=X_1\sqcup X_2 \sqcup X_3,$ where each $X_i,$ $i=1,2,3,$ is a real analytic codimension-$1$ submanifold of $\mathbb{R}_{>0}^6.$ The complement $\mathbb{R}_{>0}^6\setminus\mathcal L$ is a disjoint union of three manifolds $\Omega_i$ with boundary so that $\Omega_i\cap\partial\mathcal L=X_i,$ $i=1,2,3.$
\end{proposition}

For $\{i,j\}\subset\{1,\dots,4\},$ let $\Omega^{\pm}_{ij}=\{l\in\mathbb{R}_{>0}^6\ |\pm\phi_{ij}(l)\geqslant1 \}$ and $X^{\pm}_{ij}=\{l\in\mathbb{R}_{>0}^6\ |\phi_{ij}(l)=\pm1 \}.$ Then by Proposition \ref{lengths}, we have
$$\mathbb{R}_{>0}^6\setminus\mathcal L=\bigcup_{i\neq j}\big(\Omega^+_{ij}\cup\Omega^-_{ij}\big).$$

\begin{lemma}\label{tech}
For $\{i,j,k,h\}=\{1,2,3,4\},$ we have
\begin{enumerate}[(1)]
\item $\Omega^-_{ij}\cap \Omega^-_{ik}=\emptyset.$

\item $\Omega^-_{ij}= \Omega^-_{kh}$ and $\Omega^+_{ij}= \Omega^-_{ik}\cup \Omega^-_{ih}.$

\item $X^-_{ij}= X^-_{kh}$ and $X^+_{ij}= X^-_{ik}\cup X^-_{ih}.$
\end{enumerate}
\end{lemma}

\begin{proof}
For each $l\in \Omega^-_{ij},$ we have $\phi^i_{kh}(l)\leqslant-1$ and $x^i_{jk}+x^i_{jh}\leqslant x^i_{kh},$ which implies $\phi^i_{jk}(l)\geqslant1$ and $\phi^i_{jh}(l)\geqslant1.$ Therefore, (1) holds and $\Omega^-_{ij}\subset\Omega^+_{ik}\cap\Omega^+_{ih}.$ On the other hand, for each $l\in\Omega^+_{ik}\cap\Omega^+_{ih},$ we have $\phi^i_{jh}(l)\geqslant1$ and $\phi^i_{jk}(l)\geqslant1,$ which implies that $x^i_{jh}\leqslant|x^i_{jk}-x^i_{kh}|$ and $x^i_{jk}\leqslant|x^i_{jh}-x^i_{kh}|.$ As a consequence, $x^i_{jk}+x^i_{jh}\leqslant x^i_{kh}$ and $\phi^i_{kh}(l)\leqslant-1.$ Therefore, we have $\Omega^+_{ik}\cap\Omega^+_{ih}\subset\Omega^-_{ij},$ hence $\Omega^-_{ij}=\Omega^+_{ik}\cap\Omega^+_{ih}.$ By symmetry, we have $\Omega^-_{ij}=\Omega^+_{jk}\cap\Omega^+_{jh},$ from which we see $\Omega^-_{ij}=\Omega^+_{ik}\cap\Omega^+_{ih}\cap\Omega^+_{jk}\cap\Omega^+_{jh}=\Omega^-_{kh}.$ Now for $l\in\Omega^+_{ij},$ we have $\phi^i_{kh}(l)\geqslant1,$ which implies that $x^i_{kh}\leqslant|x^i_{jk}-x^i_{jh}|.$ If $x^i_{kh}\leqslant x^i_{jk}-x^i_{jh},$ then $\phi^i_{jh}(l)\leqslant -1$ and $l\in\Omega^-_{ik}.$ If $x^i_{kh}\leqslant x^i_{jh}-x^i_{jk},$ then $\phi^i_{jk}(l)\leqslant -1$ and $l\in\Omega^-_{ih}.$ Therefore, we have $\Omega^+_{ij}\subset \Omega^-_{ik}\cup \Omega^-_{ih}.$ On the other hand, since $\Omega^-_{ik}\subset \Omega^+_{ij}$ and $\Omega^-_{ih}\subset \Omega^+_{ij},$ we have $\Omega^-_{ik}\cup \Omega^-_{ih}\subset\Omega^+_{ij},$ from which (2) follows. (3) follows from the same argument with the inequalities replaced by equalities.
\end{proof}

\begin{proof}[Proof of Proposition \ref{degeneration}] Let $b_h=\sqrt{{2c_{ij}c_{ik}c_{jk}+c_{ij}^2+c_{ik}^2+c_{jk}^2-1}},$
 $c_{ij}=\cosh l_{ij}$ and $s_{ij}=\sinh l_{ij}$ for each $l=(l_{12},\dots,l_{34}) \in\mathbb{R}_{>0}^6$ and $l_{ij}=l_{ji}.$
 We have
\begin{equation*}
\begin{split}
\frac{\partial \phi_{ij}}{\partial l_{kh}}=-\frac{s_{ij}^2s_{kh}}{b_kb_h}\neq0,
\end{split}
\end{equation*}
which implies $\nabla \phi_{ij}\neq 0.$ Therefore, $-1$ is a regular values of $\phi_{ij}.$ By the
 Implicit Function Theorem, $X^{-}_{ij}=\phi_{ij}^{-1}(-1)$ is a smooth
 codimension-$1$ submanifold of $\mathbb{R}_{>0}^6.$ Since each $\phi_{ij}$ is real
 analytic in $\mathbb{R}_{>0}^6,$ the submanifold $X^{-}_{ij}$ is real analytic. Let $\Omega_1=\Omega^-_{12},$ $\Omega_2=\Omega^-_{13}$ and $\Omega_3=\Omega^-_{14},$ and similarly let $X_1=X^-_{12},$ $X_2=X^-_{13}$ and $X_3=X^-_{14}.$ As a consequence of Lemma \ref{tech}, we have
$$\mathbb{R}_{>0}^6\setminus\mathcal L=\Omega_1\sqcup\Omega_2\sqcup\Omega_3$$
 is a disjoint union of three $6$-dimensional submanifolds with boundary. By Lemma \ref{tech}, $X_i=\partial \Omega_i$ and $\partial \mathcal L\subset \sqcup_{i=1}^3\partial \Omega_i=\sqcup_{i=1}^3X_i.$ We claim that $X_i\subset \partial \mathcal L.$ Indeed, for each $l=(l_{12},\dots,l_{34})\in X_1,$ say, we construct a sequence $\{l^{(n)}\}\subset\mathcal L$ convergent to $l$ as follows. We let $\epsilon_n\rightarrow 0^+,$ and define $l^{(n)}=(l_{12}-\epsilon_n,l_{13},\dots,l_{34})\rightarrow l.$ Then for $n$ large enough, by (\ref{xijk}) $x^{i\ (n)}_{jk}, x^{i\ (n)}_{jh}$ and $x^{i\ (n)}_{kh}$ satisfy the triangular inequalities, and by (\ref {eq5.2}) each $\phi_{ij}(l^{(n)})\in(-1,1).$ By Proposition \ref{lengths}, $l^{(n)}\in\mathcal L$ for $n$ large enough. Therefore, we have $\partial \mathcal L=\sqcup_{i=1}^3X_i,$ which completes the proof.
\end{proof}

\begin{lemma}\label{comp0}
The function $\phi_{ij}$ extends continuously to $\mathbb{R}^6_{\geqslant0},$ and $\phi_{ij}(l)=1$ when $l_{ij}=0.$
\end{lemma}

\begin{proof} Since the denominator of $\phi_{ij}$ is never equal to $0,$ the function continuously
extends $\mathbb{R}^6_{\geqslant0}.$ Furthermore a direct
calculation show if $l_{ij}=0,$ i.e., $c_{ij}=1,$ then
$\phi_{ij}(l)=1.$  Indeed, both numerator and denominator in
(\ref{eq5.4}) are  $(c_{ik}+c_{jk})(c_{ih}+c_{jh}).$
\end{proof}

\begin{proposition}\label{degeneration2} For each subset $S$ of the edges of a tetrahedron,  let
$$\mathcal D_S=\big\{ l\in\mathbb R_{\geqslant0}^6\ \big|\ l(e)>0\text{ for }e\in S\text{ and }l(e)=0\text{ for }e\notin S\big\},$$ and let $\overline {X_i},$ $i=1,2,3,$ be the closure of $X_i$ in $\mathbb R^6_{\geqslant0}.$ If $\mathcal D_S\cap\overline{ X_i}\neq\emptyset,$ then $X^S_i\doteq\mathcal D_S\cap\overline{ X_i}$ is a real analytic codimension-$1$ submanifold of $\mathcal D_S.$
\end{proposition}

\begin{proof} Let $\overline{X^-_{ij}}$ be the closure of $X^-_{ij}$ in $\mathbb R^6_{\geqslant0}.$ We first observe that if $e_{ij}\notin S,$ then $\mathcal D_S\cap\overline{ X^-_{ij}}=\emptyset.$
Indeed, by Lemma \ref{comp0}, if $l\in\overline{X^-_{ij}},$ then $\phi_{ij}(l)=-1,$ and if $l\in\mathbb R^S,$ then $\phi_{ij}(l)=1.$ Therefore, if $l\in\mathcal D_S\cap\overline{ X^-_{ij}},$ then $e_{ij}\in S$ and hence $s_{ij}\neq 0.$ By Lemma \ref{tech}, $X^-_{ij}=X^-_{kh},$ which implies $e_{kh}\in S,$ hence $s_{kh}\neq 0.$  Letting $c_{ij},$ $s_{ij}$ and $b_i$ be as before, we have
\begin{equation*}
\begin{split}
\frac{\partial \phi_{ij}}{\partial l_{kh}}=-\frac{s_{ij}^2s_{kh}}{b_kb_h}\neq0,
\end{split}
\end{equation*}
which implies that the projection of $\nabla \phi_{ij}$ to the tangent space of $\mathcal D_S$ at $l\in\mathcal D_S\cap\overline{ X^-_{ij}}$ is non-vanishing, i.e., $\mathcal D_S$ and $\overline{X^-_{ij}}$ transversely intersect. By the Implicit Function Theorem, the intersection $\mathcal D_S\cap\overline{X^-_{ij}}=\mathcal D_S\cap\phi_{ij}^{-1}(-1)$ is a smooth codimension-$1$ submanifold of $\mathcal D_S.$ Since each $\phi_{ij}$ is real analytic, the submanifold is real analytic in $\mathcal D_S.$
\end{proof}

Let $\overline{\mathcal L}$ be the closure of $\mathcal L$ in $\mathbb R_{\geqslant0}^6$ and let $\mathcal L_S=\mathcal D_S\cap\overline{\mathcal L}.$ For $i\neq j,$ by definding $a_{ij}|_{\overline{\Omega^+_{ij}}}=0$ and $a_{ij}|_{\overline{\Omega^-_{ij}}}=\pi,$ we have

\begin{corollary}\label{dihedral}
The dihedral angle function $a_{ij}\co\mathcal L\rightarrow\mathbb{R}$ can be extended continuously
 to $\mathbb{R}_{\geqslant0}^6$ so that its extension, still denoted by
  $a_{ij}\co\mathbb{R}_{\geqslant0}^6\rightarrow\mathbb{R},$ is a constant on each component
   of $\mathcal D_S\setminus\mathcal L_S$ for each subset $S$ of the edges.
\end{corollary}

We call $a_{ij}$ in corollary \ref{dihedral} the \it dihedral
angle \rm of the generalized hyper-ideal tetrahedron $(\sigma,
\{l_{rs}\}).$

\subsection{Covolume of generalized hyper-ideal
tetrahedron}

The locally convex function covolume function $cov \co\mathcal
L\rightarrow \mathbb{R}$ defined by (\ref{2.1}) satisfies the
Schlaefli identity that
\begin{equation*}
\frac{\partial cov}{\partial l_{ij}}=a_{ij}
\end{equation*}
for $i\neq j,$ where $a_{ij}\co\mathcal L\rightarrow\mathbb{R}$ is
the dihedral angle function at the $ij$-th edge. In particular,
the differential $1$-form $\omega=\sum_{i< j}a_{ij}dl_{ij}=dcov$
is closed in $\mathcal L,$ and we can recover covolume $cov$ by
the integration $cov(l)=\int^l\omega.$

For each $l=(l_{12},\dots,l_{34})\in\mathbb R^6,$ we let
$l^+=(l^+_{12},\dots,l^+_{34})\in\mathbb R_{\geqslant0}^6$ where
$l^+_{ij}=\max\{0,l_{ij}\}.$ By Corollary \ref{dihedral},  we can
extend the function $a_{ij}\co\mathcal L\rightarrow\mathbb{R}$ to
a continuous function $a_{ij}\co\mathbb{R}^6\rightarrow\mathbb{R}$
by
$$a_{ij}(l)=a_{ij}(l^+),$$ and  call $a(l)=(a_{12}(l),\dots,a_{34}(l))$ the
\emph{dihedral angle vector} of $l\in\mathbb{R}^6.$  We define a new continuous $1$-form $\mu$ on $\mathbb{R}^6$ by
\begin{equation*}
\mu(l)=\sum_{i\neq j}a_{ij}(l)dl_{ij}.
\end{equation*}


\begin{proposition}\label{close}
The continuous differential $1$-form
$\mu=\sum_{ij}a_{ij}(l)dl_{ij}$ is closed in $\mathbb{R}^6,$ i.e.,
for any Euclidean triangle $\Delta$ in $\mathbb{R}^6,$
$\int_{\partial \Delta} \mu =0.$
\end{proposition}

\begin{proof} We prove it in two steps. In the first step, we
prove that $\mu$ is closed in $\mathbb{R}^6_{\geq 0}.$ Next, we
show $\mu$ is closed in $\mathbb{R}^6.$

By Corollary \ref{dihedral}, the differential $1$-form
$\mu=\sum_{ij}a_{ij}(l)dl_{ij}$ is continuous in $\mathbb{R}^6.$
 The restriction $\mu |_{\mathcal{L}}=\sum_{ij} a_{ij}dl_{ij}=d cov$ is closed.
 By corollary \ref{dihedral}, $a_{ij}$ is a constant in each connected component of $\mathbb{R}_{>0}^6\setminus\mathcal{L}.$
Proposition \ref{degeneration} shows that the subset $\mathcal{L}$
in $\mathbb{R}^6_{>0}$ is open and bounded
 by a smooth codimension-$1$ submanifold.  Now we use the following
 lemma.

\begin{lemma}\label{L3}(Propositions 2.4 and 2.5, \cite{Luo3})
Suppose $U \subset \mathbb{R}^N$ is an open set and
$\lambda=\sum_{i} \alpha_i(x) dx_i$ is a continuous 1-form on $U.$
\begin{enumerate}[(1)]
\item If $A\subset U $ is an open subset bounded by a smooth
codimension-$1$ submanifold of $U,$ and $\lambda|_A$ and
$\lambda|_{U\setminus\overline{A}}$ are closed, then $\mu$ is
closed in $U.$

\item If $U$ is simply connected, then $F(x)=\int ^x\lambda$ is a
$C^1$-smooth function such that
$$\frac{\partial F}{\partial x_i}=\alpha_i.$$

\item If $U$ is convex and $A\subset U$ is an open subset of $U$
bounded by a codimension-$1$ real analytic submanifold of $U$ so
that $F|U$ and $F|_{U \setminus \overline{A}}$ are locally convex,
then $F$ is convex in $U.$
\end{enumerate}
\end{lemma}

 Thus, by Lemma \ref{L3} (1), the differential $1$-form $\mu$ is closed
  in $\mathbb{R}^6_{>0}.$

For each subset $S$ of the  edges, by Equation (\ref{eq5.4}), Lemma \ref{comp0} and  a direct calculation, $\mu |_{\mathcal D_S}$ is
  closed in $\mathcal L_S.$ By definition, $\mu |_{\mathcal D_S}$ is constant in each connected component
  of $\mathcal D_S\setminus\mathcal{L}_S.$ Now Proposition \ref{degeneration2} shows that the
  subset $\mathcal{L}_S$ in $\mathcal D_S$ is open and bounded by a smooth codimension-$1$ submanifold. Thus,
   by Lemma \ref{L3} (1), the differential $1$-form $\mu |_{\mathcal D_S}$ is closed in $\mathcal D_S.$ For each Euclidean triangle
   $\Delta$  in a quadrant $Q$ of $\mathbb R^6,$ let $S$ be set of edges $e$ so that $l(e)>0$ for all $l\in Q,$ and let $\Delta_S$ be the projection of $\Delta$ to $\mathcal D_S.$ By
     definition and Lemma \ref{comp0}, $a_{ij}(l)\equiv0$ if $l_{ij}\leqslant0,$ which
     implies $\int_{\Delta}\mu=\int_{\Delta_S}\mu |_{\mathcal D_S}=0.$ As a
     consequence, $\mu$ is closed in each of the quadrants of $\mathbb{R}^6.$ Repeating
     applying  Lemma \ref{L3} (1), we conclude that $\mu$ is closed in $\mathbb{R}^6.$
\end{proof}

\begin{corollary}\label{extended}
The function $cov \co\mathbb{R}^6\rightarrow\mathbb{R}$ defined by
the integral
\begin{equation}\label{3.2}cov(l)=\int_{(0,\dots,0)}^l\mu+cov(0,\dots,0)
\end{equation}
is a $C^1$-smooth convex function.
\end{corollary}

\begin{proof}
By Lemma \ref{L3} (2) and Proposition \ref{close}, $cov$ is a
$C^1$-smooth function in $\mathbb R^6.$ By Lemma \ref{L3} (3) and
Lemma \ref{dihedral}, $cov$ is convex in $\mathbb R^6_{>0}.$ By
the continuity,  for each subset $S$ of the edges, $cov |_{\mathcal D_S}$ is convex in ${\mathcal D_S}.$ Since
$a_{ij}(l)\equiv0$ if $l_{ij}\leqslant0,$ we have
$cov(l)=cov(l^+)$ and $l^+\in\mathcal D_S$ for some $S.$ As a
consequence, $cov$ is convex in each quadrant of $\mathbb{R}^6.$
Repeat using Lemma \ref{L3} (3), we conclude that $cov$ is convex
in $\mathbb R^6.$
\end{proof}

\begin{remark}
By the work of Ushijima \cite{Ush},
$cov(0,\dots,0)=2vol(0,\dots,0)=16\Lambda(\pi/4),$ where $\Lambda$
is the Lobachevsky function defined by
$$\Lambda(a)=-\int_0^a\ln|2\sin t|dt,$$
and $vol(0,\dots,0)$ is the maximal volume amongst the generalized
hyperbolic tetrahedra.
\end{remark}

\section{Global rigidity of hyper-ideal polyhedral metrics}\label{global}

In this section, we prove Theorem \ref{main} (b) using the convex
extension of $cov$ in corollary \ref{extended}.

Let $(M,\mathcal T)$ be
 triangulated closed pseudo $3$-manifold
 and let $E=E(\T),$ $V=V(\T)$ and $T=T(\T)$ respectively be the sets of edges, vertices and tetrahedra in $\mathcal T.$
Replacing each 3-simplex in $\T$ by a hyper-ideal tetrahedron and
gluing them along codimension-1 by isometries, we obtain an
\it hyper-ideal polyhedral metric \rm on $(M, \T).$  This metric is the
same as assigning  a positive number to each edge $e \in E(\T)$ so
that each tetrahedron $\sigma$ becomes a hyper-ideal tetrahedron
with assigned numbers as edge lengths.  They are the same as
hyperbolic cone metrics on $M-N(V)$ with singularity
consisting of geodesic arcs between totally geodesic boundary.
Here $N(V)$ is an open regular neighborhood of $V$ in $M.$ We
denote by
 $\mathcal L(M,\mathcal T)$ the space of all hyper-ideal polyhedral metrics on $(M,\mathcal
 T)$ parametrized by the edge length vector $l\co E\rightarrow \mathbb R_{>0}.$ If $l \in \mathcal L(M, \T),$ its \it curvature \rm is a map $K_l:
 E(\T) \to \mathbb{R}$ sending each edge to $2\pi$ less the sum of
 dihedral angles at the edge. The \it curvature map \rm $K: \mathcal
 L(M,\T) \to \mathbb{R}^E$ sends $l$ to $K_l.$

The rigidity theorem \ref{main} (b) can be rephrased as

\begin{theorem}\label{hyper-idealrigid} For any closed triangulated
pseudo 3-manifold $(M, \T),$ a hyper-ideal polyhedral metric on
$(M, \T)$ is determined by its curvature, i.e., the curvature map
$K: \mathcal L(M, \T) \to \mathbb{R}^{E(\T)}$ is injective.

\end{theorem}

\begin{proof}   For
 each $l\in\mathcal L(M,\mathcal T)$ and each tetrahedron $\sigma\in T,$ let $l_{\sigma}\in\mathbb{R}_{>0}^6$ be
 the edge length vector of $\sigma$ in the hyperbolic polyhedral metric $l.$  Define
 the \emph{co-volume} function
\begin{equation}\label{3.1}
cov(l)=\sum_{\sigma\in T}cov(l_{\sigma})
\end{equation}
on $\mathcal{L}(M,\mathcal{T}).$ By proposition \ref{vol} (a), the
Hessian matrix of the function $cov(l_{\sigma})$ is positive
definite at each point in $\mathcal L.$ Therefore, the Hessian
matrix of $cov$ is positive definite. In particular,  $cov$ is locally
strictly convex. On the other hand, it is shown in \cite{Luo1}
that the gradient of $cov$ is $2\pi$ minus the curvature map $K_l
\in\mathbb{R}^E,$ i.e., $\nabla cov =2\pi(1,1, ..., 1)
-K_l.$  Therefore, it suffices to show that $\nabla cov$
is injective.

We extend the co-volume function $cov\co\mathcal L(M,\mathcal
T)\rightarrow\mathbb{R}$ defined by (\ref{3.1}) to a $C^1$-smooth
convex function, still denoted by
$cov\co\mathbb{R}_{>0}^E\rightarrow\mathbb{R,}$ by
\begin{equation*}
cov(l)=\sum_{\sigma\in T}cov(l_{\sigma}),
\end{equation*}  where $cov(l_{\sigma})$ is the extended convex
function given by corollary \ref{extended}.  The convexity of $cov$
follows from the fact that each summand $cov (l_{\sigma})$ is
convex.

Now suppose otherwise that there exist  $l_1\neq
l_2\in\mathcal{L}(M,\mathcal{T})$ so that $K_{l_1} = K_{l_2}.$
Joint $l_1$ and $l_2$ in $\mathbb{R}_{>0}^6$ by the line segment
$tl_1+(1-t)l_2,$ $t\in[0,1],$ and consider the convex function
$w(t)=cov(tl_1+(1-t)l_2),$ $t\in[0,1].$ By the construction,
$w\co[0,1]\rightarrow\mathbb{R}$ is a $C^1$-smooth convex function
so that $w'(t)=\nabla cov\cdot (l_2-l_1).$ Now $\nabla
cov(l_i)=(2\pi,\dots,2\pi)-K_{l_i}$ and $K_{l_1} =K_{l_2}.$ It
follows that $w'(0)=w'(1).$ Since $w$ is convex, $w(t)$ mush be a
linear function in $t.$ However, $cov$ is strictly convex near $l_1$
and $l_2.$ Therefore, $w$ is strictly convex in $t$ near $0$ and
$1.$ This is a contradiction.

\end{proof}

\section{Volume maximization of angle structures}\label{FD}

Let $(M,\mathcal{T})$ be closed triangulated pseudo 3-manifold
with set of edges $E$ and set of tetrahedra $T.$


\begin{definition}[Angle assignment of hyper-ideal type]\label{angle-hyper}  An  \emph{angle assignment} (respectively \emph{positive angle assignment}) of hyper-ideal type
on $(M, \T)$ assigns each edge $e$ in each tetrahedra $\sigma$ a
non-negative (respectively positive) number $a(e, \sigma),$ call
the dihedral angle of $e$ in $\sigma,$ so that sum of
the dihedral angles at three edges in the same tetrahedron
$\sigma$ adjacent to each vertex is less than or equal to (respectively strictly less
than) $\pi.$ The {\it cone angle} of
an angle assignment is the function $k \in
\mathbb{R}_{\geqslant 0}^E$ sending each edge $e$ to the sum of
dihedral angles at $e.$
\end{definition}
For any $k\in\mathbb R^E_{\geqslant0},$ we denote respectively  by
$\mathcal B_k^*(M,\mathcal T)$ and $\mathcal B_k(M,\mathcal T)$ the spaces of angle assignments and
positive angle assignments with cone angles $k.$
Note that if
$\mathcal B_k\neq \emptyset,$ then $\mathcal B^*_k$ is the
compact closure of $\mathcal B_k.$
When $k=(2\pi,\dots,2\pi),$ we denote respectively by $\mathcal B^*(M,\T)$ and $\mathcal B(M,\T)$ the space of the corresponding angle assignments and positive angle assignments of cone angles $2\pi.$ We note that $\mathcal B(M,\T)$ coincides with the space of {\it linear hyperbolic structures} on $(M,\T)$ defined in \cite{Luo1}.

By the work of \cite{BB} and \cite{Fu}, a positive angle
assignment $a$ on $(M, \T)$ is the same as making each tetrahedra
$\sigma \in T$ a hyper-ideal tetrahedron so that its dihedral
angles are given by $a.$  In particular, we can define the volume
of a positive angle assignment as the sum of the hyperbolic
volume of the hyper-ideal tetrahedra, i.e., the \emph{volume
function} $vol\co{\mathcal B_k(M,\mathcal
T)}\rightarrow\mathbb{R}$  is
$$vol(a)=\sum_{\sigma\in T}vol(a_{\sigma}),$$
where $a_{\sigma}\in {\mathcal B}$ is the hyper-ideal tetrahedron with
dihedral angles given by $a.$  By the work of Schlenker
\cite{Sch}, $vol$ is smooth and strictly convex on $\mathcal
B_k(M,\mathcal T).$ A theorem of Rivin\,\cite{Riv} shows that
$vol$ can be extended continuously to the compact closure
$\mathcal B_k^*(M,\mathcal T)$ of $\mathcal B_k(M,\mathcal T).$

\begin{definition}[Generalized hyper-ideal metric] We call a map $l\co E\rightarrow \mathbb R_{>0}$ a \emph{generalized hyper-ideal metric} on $(M,\T).$
\end{definition}

For each $l \in \R_{>0}^E,$ by replacing each tetrahedron $\sigma$ in
$T$ by a generalized hyper-ideal tetrahedron whose edge lengths are
given by $l,$ we define the \it dihedral angle \rm of $l$ at an
edge $e$ in a tetrahedron $\sigma>e$ to be the corresponding
dihedral angles in the generalized hyper-ideal tetrahedron whose
edge lengths are given by $l.$ The \it cone angle \rm of $l$ at $e$ is the sum of dihedral angles of $l$ at $e$ in all the tetrahedra $\sigma$ that contain $e.$ The goal of this section is to prove the following  

\begin{theorem}\label{unique2} Suppose $(M, \T)$ is a closed triangulated pseudo 3-manifold
so that  ${\mathcal B_k(M,\mathcal T)}\neq\emptyset.$ Then,

\begin{enumerate}[(a)]
\item there is a unique
 $a\in \mathcal B_k^*(M,\mathcal T)$ that achieves the maximum volume, and

\item for each generalized hyper-ideal metric
$l\in\mathbb{R}^E_{>0}$ with cone angles $k,$ the dihedral angles
$a(l)$ of $l$ is the maximum volume on $\mathcal B_k^*(M,\mathcal
T).$

\item If $a\in \mathcal B_k^*(M,\mathcal T)$ achieves the maximum
volume, then there exists a generalized hyperbolic metric $l
\in\mathbb{R}^E_{>0}$ whose dihedral angles $a(l)=a.$
\end{enumerate}

\end{theorem}


Theorem \ref{unique2} implies Theorem \ref{1.4} in \S 1. It is known to Kojima\,\cite{Ko} that every compact hyperbolic $3$-manifold with totally geodesic
 boundary admits a geometric ideal triangulation so that each tetrahedron is either hyper-ideal or flat hyper-ideal
 whose dihedral angles are $0$ and $\pi.$ One consequence of Theorem \ref{unique2} (b) is

\begin{corollary}\label{main3}
Suppose $M$ is a compact hyperbolic $3$-manifold with totally
geodesic boundary and $\mathcal T$ is a geometric ideal
triangulation of $M$ so that each tetrahedron is either
hyper-ideal or flat hyper-ideal. If $\mathcal B(M,\T)\neq\emptyset,$ then the maximum volume on
${\mathcal B^*(M,\mathcal T)}$ is equal to the hyperbolic
volume of $M.$
\end{corollary}

The counterpart of Corollary \ref{main3} for cusped hyperbolic $3$-manifolds was proved in \cite{Luo4}.


\subsection{Volume function and sub-derivatives }

Recall that $\mathcal B\subset\mathbb{R}^6$ is the space of dihedral angle
vectors of hyper-ideal tetrahedra defined by (\ref{Equation}) and let $\overline{\mathcal B}$
be its closure, i.e.,
$$\overline{\mathcal B}=\big\{ (a_{12}, \dots, a_{34})\in\mathbb R_{\geqslant0}^6\ |\ \sum_{j\neq i}a_{ij}\leqslant\pi \text{ for each } i,\text{ where } a_{ij}=a_{ji}\ \big\}.$$
By \cite{Riv}, the volume function $vol$ on $\overline{\mathcal
B}$ is continuous and convex. 
To study the volume optimization, we need to classify points in
$\overline{\mathcal B}.$

\begin{definition} We call $a=(a_{12},\dots,a_{34})\in\overline{\mathcal B}$ generalized dihedral angles
\begin{enumerate}[I.]
\item \emph{of type I} if for each $i\in\{1,\dots,4\},$
$\sum_{j\neq i}a_{ij}<\pi,$ \item\emph{of type II} if
$a=(\pi,0,0,0,0,\pi),$ $(0,\pi,0,0,\pi,0)$ or $(0,0,\pi,\pi,0,0),$
and \item \emph{of type III} if not of types I and II.
\end{enumerate}
\end{definition}

We denote by $\mathcal B_I,$ $\mathcal B_{II}$ and $\mathcal
B_{III}$ respectively the set of generalized dihedral angles of
types I, II and III. Note that $\mathcal B \subset \mathcal B_I.$
For each $\{i,j\}\subset\{1,\dots,4\},$ we let $\psi_{ij}:
\mathcal B\rightarrow\mathbb{R}$ be the function defined by
\begin{equation*}
\psi_{ij}(a)=
\frac{s_{ij}^2c_{kh}+c_{ik}c_{jk}+c_{ih}c_{jh}+c_{ij}c_{ik}c_{jh}+c_{ij}c_{ih}c_{jk}}{\sqrt{2c_{ij}c_{ik}c_{ih}+c_{ij}^2+c_{ik}^2+c_{ih}^2-1}\sqrt{2c_{ij}c_{jk}c_{jh}+c_{ij}^2+c_{jk}^2+c_{jh}^2-1}},
\end{equation*}
where $s_{ij}=\cos a_{ij}$ and  $c_{ij}=\cos a_{ij}.$ By the
Cosine Law, if $a\in\mathcal B$ is the dihedral angles of a
hyper-ideal tetrahedron $\sigma$ with
$l(a)=(l_{12}(a),\dots,l_{34}(a))$ the edge lengths, then
$l_{ij}(a)=\cosh^{-1}\psi_{ij}(a).$

\begin{lemma}\label{length}
The function $\psi_{ij}: \mathcal B\rightarrow\mathbb{R}$
continuously extends to $\mathcal B_I,$ and $\psi_{ij}(a)=1$ when
$a_{ij}=0.$
\end{lemma}

\begin{proof}
Since $\sum_{j\neq i}a_{ij}<\pi$ for $i=1,\dots,4,$ the denominator of $\psi_{ij}$ is not equal to $0,$ hence the function continuously extends. If $a_{ij}=0,$ then $c_{ij}=1$ and $\psi_{ij}(a)=
1.$
\end{proof}

For $a\in\mathcal B_{I},$ we let
$l_{ij}(a)=\cosh^{-1}\psi_{ij}(a)$ and call
$l(a)=(l_{12}(a),\dots,l_{34}(a))\in\mathbb{R}^6_{\geqslant0}$ the
\emph{associated edge lengths} of $a.$ 
Given a tetrahedron $\sigma,$ let  $S$
be a set of edges in $\sigma$ and define
$$\mathcal B_S\doteq\{ a\in\mathcal B_I\ |\ a_{\alpha}>0 \text{ for }
\alpha\in S \text{ and } a_{\alpha}=0 \text{ for } \alpha\notin
S\}.$$ The set $\mathcal B_S$ is an open convex polytope in the
smallest affine space containing it since $\mathcal B_S$ is
defined by strict linear inequalities and linear equalities.

\begin{proposition}\label{strict} For each $S,$ the restriction of the volume function  is smooth and strictly concave in $\mathcal B_S.$
\end{proposition}

\begin{proof}
By Lemma \ref{length}, the edge lengths continuously extend to
$\mathcal B_S$ with $l_{\alpha}=0$ for $\alpha\notin S.$ From the
definition, each $\psi_{\alpha}$ for $\alpha\in S$ is smooth in
$\mathcal B_S.$ As a consequence, the function
$l_{\alpha}=\cosh^{-1}\psi_{\alpha}$ is smooth and the following
differential $1$-form $\omega_S=-\frac{1}{2}\sum_{\alpha\in
S}l_{\alpha}da_{\alpha}$ is smooth in $\mathcal B_S.$ From the
definition and a direct calculation, we have $\frac{\partial
l_{\alpha}}{\partial a_{\beta}}=\frac{\partial l_{\beta}}{\partial
a_{\alpha}}$ for $\alpha,$ $\beta\in S,$ i.e., the differential
$1$-form $\omega_S$ is closed in $\mathcal B_S.$ For each
$a\in\mathcal B_S,$ take $a_0\in\mathcal B_S$ close to $a,$ and
define a function $g: \mathcal B_S\rightarrow\mathbb{R}$ by
$g(a)=\int_{a_0}^a\omega_S,$ where the path in the integral is any
smooth path in $\mathcal B_S$ connecting $a_0$ and $a.$ Since
$\omega_S$ is smooth and closed, $g$ is well defined and smooth in
$\mathcal B_S$ and $\frac{\partial g}{\partial
a_{\alpha}}=-\frac{l_{\alpha}}{2}$ for each corner $\alpha\in S.$
We claim that $g(a)=vol(a)-vol(a_0).$ The claim implies that the
volume function $vol$ is smooth in $\mathcal B_S$ and
$\frac{\partial vol}{\partial a_{\alpha}}=-\frac{l_{\alpha}}{2}$
for each $\alpha\in S.$ Now since $a_0$ is close to $a,$ we can
take a vector $w\in\mathbb{R}^6$ so that $a+w,$ $a_0+w\in\mathcal
B.$ Let $v=a-a_0$ and let
$f(t,s)\doteq-\frac{1}{2}\sum_{\alpha}l_{\alpha}(a_0+tv+sw)v_{\alpha},$
where the summation is over all the dihedral angles.  By the
continuity of $vol,$ we have $ vol(a)-vol(a_0)=\lim_{s\rightarrow
0^+}\big(vol(a+sw)-vol(a_0+sw)\big) =\lim_{s\rightarrow
0^+}\int_0^1f(t,s)dt. $ One the other hand, since $l_{\alpha}$
continuously extends to $\mathcal B_S,$ it is uniformly continuous
on the compact parallelogram determined by $a_0,$ $a,$ $a_0+w$ and
$a+w.$ As a consequence, $f$ is uniformly continuous on the
compact square $[0,1]\times[0,1]$ and
$g(a)=\int_0^1f(t,0)dt=\int_0^1\lim_{s\rightarrow 0^+}f(t,s)dt.$
Again, since $f$ is uniformly continuous, we can switch the order
of taking limit and integrating, and we have
$g(a)=vol(a)-vol(a_0).$

Now we  show the strict concavity of $vol$ in $\mathcal
B_{\mathcal S}.$ Let $l: \mathcal
B_S\rightarrow\mathbb{R}^6_{\geqslant0}$ be map defined by the
restriction of the associated edge lengths and let $\mathcal
L_S=l(\mathcal B_S).$ From the definition, each $\phi_{\alpha}$
for $\alpha\in S$ is smooth in $\mathcal L_S,$ hence the function
$a_{\alpha}=\cos^{-1}\phi_{\alpha}$ is smooth in $\mathcal L_S.$
Now we have two smooth maps $a=(a_{\alpha}): \mathcal
L_S\rightarrow\mathcal B_S$ and $l=(l_{\alpha}): \mathcal
B_S\rightarrow\mathcal L_S$ which are, by the Cosine Law, inverses
of each other. As a consequence, the map $l: \mathcal
B_S\rightarrow\mathcal L_S$ is a local diffeomorphism and the
Jocobi matrix $[{\partial l_{\alpha}}/{\partial
a_{\beta}}]_{\alpha,\beta\in S}$ is non-singular. Since $\mathcal
B_S$ is connected, the signature of the Hessian matrix
$-\frac{1}{2}[{\partial l_{\alpha}}/{\partial
a_{\beta}}]_{\alpha,\beta\in S}$ of the volume function $vol$ on
$\mathcal B_S$ is independent of the choice of $a\in\mathcal B_S.$
By a direct calculation, using the formula in
Guo\,(\cite{Guo}, Theorem 1), the matrix $-\frac{1}{2}[{\partial
l_{\alpha}}/{\partial a_{\beta}}]_{\alpha,\beta\in S}$ is negative
definite at $a_{\alpha}=\pi/4$ for each $\alpha\in S.$ This
implies that $vol$ is locally strictly concave in $\mathcal B_S.$
Since $\mathcal B_S$ is convex, the volume function $vol$ is
strictly concave in $\mathcal B_S.$
\end{proof}

\begin{lemma}\label{lemma6.7} Let $l\in\mathbb R_{>0}^6$ be the edge length vector of a generalized hyper-ideal tetrahedron with dihedral angles $a(l)\in{\mathcal B_{II}},$ and let  $v\in\mathbb R^6$ so that  $a(l)+v\in\mathcal B.$
Then
$$\lim_{t\rightarrow0^+}\frac{d}{dt}vol(a(l)+tv)\leqslant-\frac{1}{2}v\cdot
l.$$


\end{lemma}


\begin{proof}  Let $f\co[0,1]\rightarrow\mathbb R$ be the function defined by $f(t)=vol(a(l)+tv).$ For (1), by the concavity of $vol$ and the Mean Value Theorem, we have
$f'(t)<\frac{f(t)-f(0)}{t}$ for all $t\in(0,1).$ Since $a(l)\in\mathcal B_{II},$
$f(0)=vol(a(l))=0,$ and
\begin{equation}\label{label}
\begin{split}
f'(t)<\frac{f(t)}{t}.
\end{split}
\end{equation}
By Proposition \ref{degeneration}, $\mathcal{L}$ is open in
$\mathbb{R}^6_{>0}$ and $\partial\mathcal{L}$ is a smooth codimension--$1$
submanifold, hence for each
$l\in\partial\mathcal L,$ there exists a sequence
$\{l^{(n)}\}\subset\mathcal{L}$ converging to $l$ and the
corresponding dihedral angles $\{a^{(n)}\}$ converging to $a(l).$
Since the volume function $vol$ is strictly concave in $\mathcal
B,$ we have for all $n$ and for all $b\in\mathcal B$ that
$vol(b)-vol(a^{(n)})<\nabla vol(a^{(n)})\cdot(b-a^{(n)}).$
 By Schlaefli formula,
$vol(b)-vol(a^{(n)})<-\frac{l^{(n)}}{2}\cdot(b-a^{(n)}).$
Since $vol$ is continuous on $\overline{\mathcal B}$ and
$vol(a(l))=0,$ as $n$ approaches $+\infty,$
$vol(b)\leqslant-\frac{l}{2}\cdot(b-a(l)).$
In particular, when $b=a(l)+tv,$ we have
\begin{equation}\label{label2}
f(t)\leqslant-\frac{t}{2}v\cdot l.
\end{equation}
By (\ref{label}) and (\ref{label2}), we have
$f'(t)<-\frac{1}{2}v\cdot l,$
hence
$\lim_{t\rightarrow0^+}f'(t)\leqslant-\frac{1}{2}v\cdot l.$ For $l=(l_{12},\dots,l_{34})\notin\overline{\mathcal L},$ by Proposition \ref{lengths}, there exists an $m=(m_{12},l_{13},\dots,l_{34})\in\partial\mathcal L$ such that $a(m)=a(l).$ By the previous case, we have
\begin{equation}\label{label3}
\lim_{t\rightarrow0^+}f'(t)\leqslant-\frac{1}{2}v\cdot m.
\end{equation}
 Without loss of generality, we may assume that $a(l)=(\pi,0,0,0,0,\pi).$ By Cosine Law, $l_{12}>m_{12}.$ Since $a(l)+v\in\mathcal B,$ we have $\pi+v_{12}=(a(l)+v)_{12}<\pi,$ hence $v_{12}<0.$ As a consequence, $v\cdot l-v\cdot m=v_{12}(l_{12}-m_{12})<0.$ Combined with (\ref{label3}), we have $\lim_{t\rightarrow0^+}f'(t)<-\frac{1}{2}v\cdot l.$
\end{proof}

\begin{lemma}\label{lemma} Let $a\in{\mathcal B_{III}}.$ Then for each $v\in\mathbb{R}^6$ so that $a+v\in\mathcal B,$ one has
$$\lim_{t\rightarrow0^+}\frac{d}{dt}vol(a+tv)=+\infty.$$
\end{lemma}

\begin{proof}
We let $s_{ij}(t)=\sin(a_{ij}+tv_{ij}),$ $c_{ij}(t)=\cos(a_{ij}+tv_{ij})$ and
$$u_{ij}(t)=s_{ij}^2(t)c_{kh}(t)+c_{ik}(t)c_{ih}(t)+c_{jk}(t)c_{jh}(t)
+c_{ij}(t)c_{ik}(t)c_{jh}(t)+c_{ij}(t)c_{ih}(t)c_{jk}(t)$$
for $\{i,j\}\subset\{1,\dots,4\},$
and let
$$b_i(t)=\sqrt{2c_{ij}(t)c_{ik}(t)c_{ih}(t)+c_{ij}^2(t)+c_{ik}^2(t)+c_{ih}^2(t)-1}$$ for $i\in\{1,\dots,4\}.$
 Then
$\cosh l_{ij}(a+tv)=\frac{u_{ij}(t)}{b_i(t)b_j(t)}.$
Let $m_{ij}\geqslant0$ and $n_i\geqslant0$ respectively be the asymptotic orders of $u_{ij}(t)$ and $b_i(t)$ as $t$ approaches $0,$ i.e.,
$u_{ij}(t)=u_{ij}t^{m_{ij}}+o(t^{m_{ij}})$
and
$b_i(t)=b_{i}t^{n_i}+o(t^{n_i})$
for some constants $u_{ij}\neq0$ and $b_i\neq0.$ As $t$ approaches $0,$  $\cosh(l_{ij}(a+tv))=\frac{u_{ij}}{b_ib_j}t^{m_{ij}-n_i-n_j}+o(t^{m_{ij}-n_i-n_j}),$
and by Schlaefli formula,
\begin{equation*}
\begin{split}
\lim&_{t\rightarrow 0^+}\Big(\frac{d}{dt}vol(a+tv)-\big(-\frac{1}{2}\sum_{i\neq j}v_{ij}(m_{ij}-n_i-n_j)\big)\ln t\Big)\\
=&-\frac{1}{2}\sum_{i\neq j}v_{ij}\lim_{t\rightarrow 0}\Big(l_{ij}(a+tv)-(m_{ij}-n_i-n_j)\ln t\Big)=-\frac{1}{2}\sum_{i\neq j}v_{ij}\ln\frac{u_{ij}}{ b_ib_j}
\end{split}
\end{equation*}
is a finite number. Therefore, to prove the result, it suffices to
prove that the $-\frac{1}{2}\sum_{i\neq
j}v_{ij}(m_{ij}-n_i-n_j)$ is strictly negative. First we consider
the case that, up to a permutation of the vertices,
$a\neq(0,\alpha,\pi-\alpha,\pi-\alpha,\alpha,0)$ for some
$a\in(0,\pi)$ and $a\neq(\pi,0,0,0,0,\beta)$ for some
$\beta\in[0,\pi).$ In this case, if $a_{ij}=\pi$ for some
$\{i,j\}\subset\{1,\dots,4\},$ then
$a_{ik}=a_{ih}=a_{jk}=a_{jh}=0,$ which was ruled out by the
assumption. If $a_{ij}\in(0,\pi)$ for some
$\{i,j\}\subset\{1,\dots,4\},$ then we claim that $m_{ij}=0,$
i.e., $u_{ij}(0)\neq0.$ Indeed, letting $s_{ij}=\sin a_{ij}$ and
$c_{ij}=\cos a_{ij},$ we have
$u_{ij}(0)=s_{ij}^2(c_{kh}+c_{ik}c_{jk})+(c_{ih}+c_{ij}c_{ik})(c_{jh}+c_{ij}c_{jk})$
and also
$u_{ij}(0)=s_{ij}^2(c_{kh}+c_{ih}c_{jh})+(c_{ik}+c_{ij}c_{ih})(c_{jk}+c_{ij}c_{jh}).$
Since $a\in\overline{\mathcal B},$ we have
$c_{pq}+c_{pr}c_{ps}\geqslant0,$ $\{p,q,r,s\}=\{1,2,3,4\},$ and
the equality holds if and only if $a_{pq}+a_{pr}+a_{ps}=\pi$ and
one of $a_{pr}$ and $a_{ps}$ equals $0.$ Therefore, since
$s^2_{ij}>0,$ $u_{ij}(0)=0$ only if $c_{kh}+c_{ik}c_{jk}=0,$
$c_{kh}+c_{ih}c_{jh}=0$ and one of $c_{ih}+c_{ij}c_{ik}$ and
$c_{jh}+c_{ij}c_{jk}$ equals $0.$ If, say,
$c_{jh}+c_{ij}c_{jk}=0,$ then from the last equation, we have
$a_{jk}=0$ and $a_{jh}=\pi-a_{ij}\in(0,\pi).$ With
$c_{kh}+c_{ih}c_{jh}=0,$ we have $a_{ih}=0$ and
$a_{kh}=\pi-a_{jh}=a_{ij}.$ From $c_{kh}+c_{ik}c_{jk}=0$ and
$a_{jk}=0,$ we have $a_{ik}=\pi-a_{kh}=\pi-a_{ij}.$ As a
consequence, up to a permutation of indices,
$a=(0,a_{ij},\pi-a_{ij},\pi-a_{ij},a_{ij},0)$ with
$a_{ij}\in(0,\pi),$ which was ruled out by the assumption. Hence
the claim is true. Now for $a_{ij}=0,$ since $a+v\in\mathcal B,$
we have $v_{ij}=a_{ij}+v_{ij}>0$ and
$-\frac{1}{2}v_{ij}{(m_{ij}-n_i-n_j)}\leqslant
\frac{1}{2}v_{ij}(n_i+n_j).$ By the definition of $b_i(t),$ we
have $n_i>0$ if and only if $\sum_{j\neq i}a_{ij}=\pi.$ Since $a$
is of type III, there exists at least one $i\in\{1,\dots,4\}$ with
$\sum_{j\neq i}a_{ij}=\pi.$ For such $i,$ since $a+v\in\mathcal
B,$ we have $\sum_{j\neq i}v_{ij}< 0.$ Thus,
$-\frac{1}{2}\sum_{i\neq
j}v_{ij}(m_{ij}-n_i-n_j)\leqslant\frac{1}{2}\sum_{i\neq
j}v_{ij}(n_i+n_j)=\frac{1}{2}\sum_{i=1}^4(\sum_{j\neq
i}v_{ij})n_i<0.$ The two sporadic cases are verified by a direct
calculation. In the case that
$a=(0,\alpha,\pi-\alpha,\pi-\alpha,\alpha,0)$ with $a\in(0,\pi),$
we have $n_i=2$ for each $i\in\{1,\dots, 4\},$
$n_{13}=n_{14}=n_{23}=n_{24}=1$ and $n_{12}=n_{34}=2.$ Since
$a_{12}=a_{34}=0$ and $\sum a_{ij}=2\pi,$ we have $v_{12}>0,$
$v_{34}>0$ and $\sum v_{ij}<0.$ Therefore, $-\frac{1}{2}\sum
v_{ij}(m_{ij}-n_i-n_j)
=v_{12}+v_{34}+\frac{3}{2}(v_{13}+v_{14}+v_{23}+v_{24})<\,\frac{3}{2}\sum
v_{ij}<0.$ In the case that $a=(\pi,0,0,0,0,\beta)$ with
$\beta\in[0,\pi),$ we have $n_1=n_2=2,$ $n_3=n_4=0,$
$m_{12}=m_{13}=m_{14}=m_{23}=m_{24}=2$ and $m_{34}=0.$  Since
$a_{12}=\pi,$ we have $v_{12}<0$ and $-\frac{1}{2}\sum
v_{ij}(m_{ij}-n_i-n_j)=v_{12}<0.$
\end{proof}


\subsection{Fenchel duality}

Let $(M,\mathcal T)$ be closed triangulated pseudo 3-manifold. The
space of all generalized hyper-ideal polyhedral metrics on $(M,
\T),$ parameterized by the edge length vectors, is
$\mathbb{R}^{E}.$
The
\emph{co-volume} function $cov\co\mathbb{R}^E\rightarrow\mathbb{R}$
is defined by
$$cov(l)=\sum_{\sigma\in T} cov(l_{\sigma}),$$
where each summand is the function defined by (\ref{3.2}). By
Corollary \ref{extended}, $cov$ is $C^1$-smooth and convex, hence
its Fenchel dual function $$cov^*(k)=\sup\{k\cdot l-cov(l)\ |\
l\in\mathbb{R}^E\}$$ is well-defined, convex and lower
semicontinuous in $\mathbb{R}^E.$ The goal of this subsection is
to show that the Fenchel dual function $cov^*(k)$ optimizes the
volume function on the space of non-positive angle assignments of
given cone angles.

\begin{theorem}\label{FD3}
Let $$\mathcal D(\T)=\big\{k\in\mathbb{R}^E\ |\
\mathcal B_k^*(M,\mathcal T)\neq\emptyset\big\}$$ and let
$U\co\mathbb{R}^E\rightarrow \mathbb{R}$ be the function defined
by
\begin{equation*}
U(k)=\left\{
\begin{array}{cl} \min\{-2vol(a)\ |\ a\in{\mathcal B_k^*(M,\mathcal T)}\} & \text{ if }k\in {\mathcal D(\T)},\\
+\infty & \text{ if }k\notin {\mathcal D(\T)}.
\end{array} \right.
\end{equation*}
Then $cov^*(k)=U(k)$ for all $k\in\mathbb{R}^E.$
\end{theorem}


 The proof of
Theorem \ref{FD3} relies on the following propositions.

\begin{proposition}\label{continuous2}
$U$ is convex and continuous in ${\mathcal D(\T)}.$
\end{proposition}

\begin{proof} The proof follows by the same argument of Proposition \ref{continuous}.\end{proof}

\begin{proposition}\label{image2}
For each $k\in \mathbb R^E$ so that $\mathcal B_k(M, \mathcal T)\neq\emptyset,$ the function $cov_k\co\mathbb{R}^E\rightarrow\mathbb{R}$ defined by
$${cov}_k(l)=cov(l)-k\cdot l$$ has a critical point. Moreover, all the critical points of $cov_k$ are in $\mathbb R_{>0}^E.$
\end{proposition}

\begin{proof} We will use the method developed by Colin de Verdi\`ere's \cite{CV}. Take any $a\in\mathcal
B_k(M,\mathcal T).$ We rewrite $cov_k$ as
$$cov_k(l)=\sum_{\sigma\in T}\big(cov(l_{\sigma})- a_{\sigma}\cdot l_{\sigma}\big).$$
For each $\sigma\in T,$ we let $l(a_\sigma)\in\mathbb R^6$ be the
edge length vector of the hyper-ideal tetrahedron whose dihedral angles
are $a_{\sigma}.$  Then $l(a_\sigma)$ is the unique
critical point of the convex function
$cov_{\sigma,k}\co\mathbb{R}^6\rightarrow\mathbb{R}$ defined by
$$cov_{\sigma,k}(l)=cov(l)-a_{\sigma}\cdot l.$$
Since $cov_{\sigma,k}$ is strictly convex near $l_{\sigma},$ the
function $cov_{\sigma,k}$ is \it closed \rm in $\mathbb{R}^6,$
i.e., $\lim_{|l|\mapsto+\infty}cov_{\sigma,k}(l)=+\infty.$ As a
consequence, the function $cov_k$ is closed and convex in
$\mathbb{R}^E.$ This shows that $cov_k$ has a critical point $l$
in $\mathbb{R}^E.$ Moreover, since $\mathcal B_k(M,\mathcal
T)\neq\emptyset,$ $k(e)>0$ for each edge $e.$ This implies that
$l(e)>0$ for each $e\in E,$ i.e., $l\in\mathbb R_{>0}^E.$ Indeed,
if otherwise $l(e)\leqslant0$ for some $e\in E,$ then by Lemma
\ref{comp0}, all the dihedral angles of $l$ at $e$ are zero, hence
$k(e)=0,$ which is a contradiction.
 \end{proof}

As consequences of Proposition \ref{image2}, we have

\begin{proposition}\label{image} If $\mathcal B_k(M, \mathcal T)\neq\emptyset,$ then there exists a generalized hyper-ideal metric $l\in\mathbb{R}_{>0}^E$ such that the dihedral angles $a(l) \in \mathcal B_k^*(M,\mathcal T).$
\end{proposition}

\begin{proof} Any critical point of $cov_k$ satisfies the desired condition.
\end{proof}

\begin{proposition}\label{M4} The image $K(\mathcal L(M,\mathcal T))\cap(\pi,2\pi)^E$ is a convex open polytope in $\mathbb R^E.$
\end{proposition}

\begin{proof}Denoting by $\mathcal K(\mathcal T)$ the convex open polytope $$\{(2\pi,\dots,2\pi)-k\ |\ \mathcal B_k(M,\mathcal T)\neq\emptyset\}.$$ We claim that $K(\mathcal L(M,\mathcal T))\cap(\pi,2\pi)^E=\mathcal K(\mathcal T)\cap (\pi,2\pi)^E.$ Indeed, since $K(\mathcal L(M,\mathcal T))\subseteq\mathcal K(\mathcal T),$ we have that $K(\mathcal L(M,\mathcal T))\cap(\pi,2\pi)^E\subseteq \mathcal K(\mathcal T)\cap(\pi,2\pi)^E.$ On the other hand, by Proposition \ref{image}, for each $k\in\mathcal K(\mathcal T)\cap(\pi,2\pi)^E,$ there exists an $l\in\mathbb R_{>0}^E$ such that $K(l)=k.$ It now suffices to show that $l\in\mathcal L(M,\mathcal T).$ Since $k\in(\pi,2\pi)^E,$ the cone angle of $l$ at each edge $e$ is in the range $(0,\pi).$ As a consequence, all the dihedral angles of $\mathcal T$ in $l$ are in the range $(0,\pi).$ Thus, all the tetrahedra of $\mathcal T$ are hyper-ideal in $l,$ and $l$ is in $\mathcal L(M,\mathcal T).$
\end{proof}

There are examples showing that the whole image $K(\mathcal L(M,\mathcal T)),$ or the subset $K(\mathcal L(M,\mathcal T))\cap(0,2\pi)^E,$ is in general neither convex nor a polytope in $\mathbb R^E.$

\begin{proof}[Proof of Theorem \ref{FD3}]

We first show that if $k\notin{\mathcal D(\T)},$ then $cov^*(k)>C$ for all $C>0.$ Since the space $ \overline{\mathcal B}^T$ of all possible dihedral angles is compact and $vol$ is continuous, there exists a constant $C_1>0$ so that
$vol(a)\leqslant C_1$
for all $a\in \overline{\mathcal B}^T.$ Since $\{k\}$ and ${\mathcal D(\T)}$ are compact and convex in $\mathbb{R}^E,$ by the Separation Theorem of Convex Sets, there exists an $l_0\in\mathbb R^E$ so that
$$k\cdot l_0-c\cdot l_0>C+2C_1$$
for all $c\in{\mathcal D(\T)}.$
In particular, letting $c(l_0)\in{\mathcal D(\T)}$ be the cone angle vector of $l_0,$ we have
$$k\cdot l_0-c(l_0)\cdot l_0>C+2C_1.$$
Therefore,
\begin{equation*}
\begin{split}
cov^*(k)&\geqslant k\cdot l_0-cov(l_0)\\
&=k\cdot l_0-c(l_0)\cdot l_0-2vol(a(l_0))\\
&> C+2C_1-2C_1=C.
\end{split}
\end{equation*}

Now we prove that $cov^*(k)=U(k)$ in $\mathcal D(\T).$
By Proposition \ref{image2}, if $\mathcal B_k(M,\T)\neq\emptyset,$ the function $cov_k$ has a critical
point  $l\in\mathbb{R}_{>0}^E,$ and $cov^*(k)=-cov_k(l)=-2vol(a(l)).$
To show that $U(k)=-2vol(a(l)),$ i.e., $a(l)$ achieves the maximum
volume, it suffices to show that the sub-derivative
 $$\lim_{t\rightarrow0^+}\frac{d}{dt}vol((1-t)a(l)+tb)\leqslant0$$
 for each $b\in\mathcal B_k(M,\mathcal{T}).$ Let $v=b-a(l)$ and let $v_{\sigma}=b_{\sigma}-a_{\sigma}(l)$ for each $\sigma\in T.$
We have
$\sum_{\sigma\supset e}b(e,\sigma)=\sum_{\sigma\supset e}a(l)(e,\sigma)=k(e),$
hence
$\sum_{\sigma\supset e}v(e,\sigma)=0$
for each edge $e.$ From this, we see
$\sum_{\sigma\in T}v_{\sigma}\cdot l_{\sigma}=\sum_{e\in E}\big(\sum_{\alpha<e}v_{\alpha}\big)l(e)=0.$
Let $ S$ be the subset of ${T}$ consisting of hyper-ideal
tetrahedra in $l$ and let $ F={T}\setminus  S.$ We have
$-\sum_{\sigma\in S}v_{\sigma}\cdot l_{\sigma}=\sum_{\sigma\in
F}v_{\sigma}\cdot l_{\sigma}.$ Then by Schlaefli formula,
\begin{equation*}
\begin{split}
\lim_{t\rightarrow0^+}\frac{d}{dt}vol(a(l)+tv)&=\sum_{\sigma\in F}\lim_{t\rightarrow0^+}\frac{d}{dt}vol(a_{\sigma}(l)+tv_{\sigma})-\frac{1}{2}\sum_{\sigma\in S}v_{\sigma}\cdot l_{\sigma}\\
&=\sum_{\sigma\in {
F}}\Big(\lim_{t\rightarrow0^+}\frac{d}{dt}vol(a_{\sigma}(l)+tv_{\sigma})+\frac{1}{2}v_{\sigma}\cdot
l_{\sigma}\Big)\leqslant0,
\end{split}
\end{equation*}
where the inequality is from
Lemma \ref{lemma6.7}. Hence $cov^*$ and $U$ coincide on the subset $\{k\ |\ \mathcal B_k(M,\T)\neq\emptyset\}$ of $\mathcal D(\T)$ which contains the relative interior of $\mathcal D(\T).$ Since $cov^*$ and $U$ are convex and lower semicontinuous, by Lemma \ref{relativeint}, $cov^*(k)=U(k)$ on $\mathcal D(\T).$
\end{proof}


\subsection{Proofs of  the main results}

\subsubsection{A proof of Theorem \ref{unique2}}
For (a), suppose otherwise that there exist $a^{(1)}\neq a^{(2)}\in{\mathcal
B_k^*(M,\mathcal T)}$ that achieve the maximum volume. Connect
$a^{(1)}$ and $a^{(2)}$ by the line segment
$L(t)=ta^{(1)}+(1-t)a^{(2)},$ $t\in[0,1],$ and consider the
concave function $f(t)=vol(ta^{(1)}+(1-t)a^{(2)}).$ By the
maximality of $a^{(i)},$ the function $vol$ is constant in
$[0,1].$ On the other hand, we let
$a^{(i)}_{\sigma}\in\overline{\mathcal B}$ be the restriction of
$a^{(i)}$ to $\sigma$ and let
$f_{\sigma}(t)=vol(ta^{(1)}_{\sigma}+(1-t)a^{(2)}_{\sigma})$ for
each $\sigma\in T.$ Then by the maximality and Lemma \ref{lemma},
each $L_{\sigma}(t)\doteq ta^{(1)}_{\sigma}+(1-t)a^{(2)}_{\sigma}$
is not of type III. As a consequence, the interior of the line
segment $L_{\sigma}$ lies in $\mathcal B_I$ or $\mathcal B_{II}.$  Since $\mathcal B_{II}$ is discrete in $\overline{\mathcal B}$ and $a^{(1)}\neq a^{(2)},$ there is at least one $\sigma_0\in T$ such that $L_{\sigma_0}$ lies in $\mathcal B_I.$ We claim that the
interior of $L_{\sigma_0}$ lies in $\mathcal B_{\mathcal S}$ for
some subset $\mathcal S$ of the edges of $\sigma_0.$ Indeed, if
$L_{\sigma_0}(t_1)\in\mathcal B_{\mathcal S_1}$ and
$L_{\sigma_0}(t_2)\in\mathcal B_{\mathcal S_2}$ for some $t_1,$
$t_2\in(0,1),$ then the interior of $L_{\sigma_0}$ lies in $\mathcal
B_{\mathcal S_1\cup\mathcal S_2}.$ By Proposition \ref{strict},
the function $f_{\sigma_0}$ is strictly concave, hence
$f=\sum_{\sigma\in T}f_{\sigma}$ is strictly concave in
$[0,1],$ which is a contradiction.
For (b), by the assumption, $l$ is a critical point of the
function $cov_k,$ and $cov_k(l)=2vol(a(l)).$ By Theorem \ref{FD3},
$cov_k(l)=-cov^*(k)=-U(k)=max\{2vol(a)\ |\ a\in{\mathcal
B_k^*(M,\mathcal T)}\}.$
For (c), by Proposition \ref{image}, there
exists an $l\in\mathbb{R}_{>0}^E$ such that
$a(l)\in{\mathcal B_k^*(M,\mathcal T)}.$ By (b),
$a(l)$ achieves the maximum volume, and by (a),
$a(l)=a.$



\subsubsection{A proof of Corollary \ref{main3}}
Let $l\in\mathbb{R}_{>0}^E$ be the edge length vector in the hyperbolic
metric of $(M,\T)$ for which $\T$ is geometric. Then the dihedral angles  $a(l)\in{\mathcal B^*(M,\mathcal T)}.$ By Theorem
\ref{unique2} (b), $a(l)$ achieves the maximum volume on
${\mathcal B^*(M,\mathcal T)}.$ Since the
triangulation $\mathcal T$ is geometric, $vol(a(l))$ equals the
hyperbolic volume of $M.$


\noindent
\noindent
Feng Luo\\
Department of Mathematics, Rutgers University\\
New Brunswick, NJ 08854, USA\\
(fluo@math.rutgers.edu)
\\

\noindent
Tian Yang\\
Department of Mathematics, Stanford University\\
Stanford, CA 94305, USA\\
(yangtian@math.stanford.edu)

\begin{thebibliography}{99}

\bibitem {B} Bonahon, Francis, {\em A Schl\"afli-type formula for convex cores of hyperbolic 3-manifolds}. J. Differential Geom. 50 (1998), no. 1, 25--58.

\bibitem{BB} Bao, Xiliang; Bonahon, Francis,  {\em Hyperideal polyhedra in hyperbolic 3-space}. Bull. Soc. Math. France 130 (2002), no. 3, 457--491.

\bibitem{BL} Borwein, Jonathan; Lewis, Adrian, {\em Convex Analysis and Nonlinear Optimization: Theory and Examples} 2006, (2 ed.). Springer.

\bibitem{BPS} Bobenko, Alexander; Pinkall, Ulrich; Springborn, Boris, {\em Discrete conformal maps and ideal hyperbolic polyhedra}. arXiv:1005.2698.

\bibitem{Casson} A. Casson, private communication.


\bibitem{choi} Choi, Young-Eun, {\em Positively oriented ideal triangulations
on hyperbolic three-manifolds}. Topology 43 (2004), no. 6, 1345--1371.


\bibitem{CKP} Cohn, Henry; Kenyon, Richard; Propp, James; {\em
A variational principle for domino tilings}. J. Amer. Math. Soc. 14
(2001), no. 2, 297--346 (electronic).


\bibitem{CV} Colin de Verdi\`ere, Yves, {\em Un principe variationnel pour les empihements de cercles}. Invent. Math. 104 (1991) 655--669.

\bibitem{Fen1} Fenchel, Werner, {\em On conjugate convex functions}. Canadian J. Math. 1 (1949), 73--77.

\bibitem{Fen2} --------, {\em Convex cones, sets, and functions}. 1951, Mimeographed Notes, Princeton University, Princeton, New Jersey.

\bibitem{FI} Fillastre, Francois; Izmestiev, Ivan, {\em Hyperbolic cusps with convex polyhedral boundary}. Geom. Topol. 13 (2009), 457--492.

\bibitem{FI2} --------, {\em Gauss images of hyperbolic cusps with convex polyhedral boundary}. Trans. Amer. Math. Soc. 363 (2011), no. 10, 5481--5536.

\bibitem{FP} Frigerio, Roberto; Petronio, Carlo, {\em Construction and recognition of hyperbolic 3-manifolds with geodesic boundary}. Trans. Amer. Math. Soc. 356 (2004), no. 8, 3243--3282.

\bibitem{Fu} Fujii, Michihiko, {\em Hyperbolic 3-manifolds with totally geodesic boundary}. Osaka J. Math. 27 (1990), no. 3, 539--553.

\bibitem{GF} Gu\'eritaud, Francois; Futer, David (appendix), {\em On canonical triangulations of once- punctured torus bundles and two-bridge link complements}, Geom. Topol. 10 (2006), 1239--1284.

\bibitem{Guo} Guo, Ren, {\em Calculus of generalized hyperbolic tetrahedra}. Geom. Dedicata 153 (2011), 139--149.

\bibitem{HK} Hodgson, Craig; Kerckhoff, Steven, {\em Rigidity of hyperbolic cone-manifolds and hyperbolic Dehn surgery}. J. Differential Geom. 48 (1998), no. 1, 1--59.

\bibitem{I} Izmestiev, Ivan, {\em Infinitesimal rigidity of convex polyhedra through the second derivative of the Hilbert-Einstein functional}. To appear in Canad. J. Math. 

\bibitem{Ko} Kojima, Sadayoshi, {\em Polyhedral decomposition of hyperbolic 3-manifolds with totally geodesic boundary}. Aspects of low-dimensional manifolds, 93--112,
Adv. Stud. Pure Math., 20, Kinokuniya, Tokyo, 1992.

\bibitem{La1} Lackenby, Marc, {\em Word hyperbolic Dehn surgery}. Invent. Math. 140 (2000), no. 2, 243--282.

\bibitem{Luo1} Luo, Feng, {\em A combinatorial curvature flow for compact 3-manifolds with boundary}. Electron. Res. Announc. Amer. Math. Soc. 11 (2005), 12--20 (electronic).

\bibitem{Luo2} --------, {\em Volume and angle structures on 3-manifolds}. Asian J. Math. 11 (2007), no. 4, 555--566.

\bibitem{Luo3} --------, {\em Rigidity of polyhedral surfaces, III }. Geom. Topol. 15 (2011), no. 4, 2299--2319.

\bibitem{Luo4} --------, {\em A note on complete hyperbolic structures on ideal triangulated $3$-manifolds}. Topology and geometry in dimension three, 19--26, Contemp. Math., 560, Amer. Math. Soc., Providence, RI, 2011.


\bibitem{Luo5} --------, {\em Volume optimization, normal surfaces, and Thurston's equation on triangulated 3-manifolds,}  J. Differential Geom. 93 (2013), no.2, 299--326.

\bibitem{LY} Luo, Feng; Yang, Tian, {\em Volume and rigidity of hyperbolic polyhedral $3$-manifolds II}, in preparation.



\bibitem{MM} Mazzeo, Rafe; Montcouquiol, Gr\'egoire, {\em Infinitesimal rigidity of cone-manifolds and the Stoker problem for hyperbolic and Euclidean polyhedra}, J. Differential Geom. 87 (2011), no. 3, 525--576.

\bibitem{M} Montcouquiol, Gr\'egoire, {\em Deformations of hyperbolic convex polyhedra and cone-3-manifolds}, to appear in Geom. Dedicata (2013).


\bibitem{neumann} Neumann, Walter, {\em Combinatorics of triangulations and the Chern-Simons invariant for hyperbolic 3-manifolds}, Topology '90 (Columbus, OH, 1990), 243--271, Ohio State Univ.
Math. Res. Inst. Publ., 1, de Gruyter, Berlin, 1992.

\bibitem{P} Penner, Robert C, {\em The decorated Teichm\"{u}ller
space of punctured surfaces}.  Comm. Math. Phys. 113  (1987), MR 0919235.
299--339.

\bibitem{QYZ} Qiu, Ruifeng; Yang, Tian; Zhang, Faze, {\em Angle structures and hyperbolic $3$-manifolds with totally geodesic boundary}, in preparation.

\bibitem{roc} Rockafellar, R. Tyrrell, {\em Convex analysis.}
Princeton Mathematical Series, No. 28 Princeton University Press, Princeton, N.J. 1970

\bibitem{Riv2} Rivin, Igor, {\em Euclidean structures on simplicial surfaces and hyperbolic volume}, Ann. of Math. (2) 139 (1994), no. 3, 553--580.

\bibitem{Riv1} --------, {\em Combinatorial optimization in geometry}. Adv. in Appl. Math. 31(2003), no. 1, 242--271.

\bibitem{Riv} --------, {\em Volumes of degenerating polyhedra -- on a conjecture of J. W. Milnor}. Geom. Dedicata (2008) 131, 73--85.

\bibitem{Sch} Schlenker, Jean-Marc, {\em Hyperideal polyhedra in hyperbolic manifolds}. arXiv: math/0212355.

\bibitem{Thurston}  Thurston, William, {\em The geometry and topology of 3--manifolds}, Princeton Univ. Math. Dept. (1978).

\bibitem{Ush1} Ushijima, Akira, {\em A unified viewpoint about geometric objects in hyperbolic space and the generalized tilt formula}. Hyperbolic spaces and related topics, II (Japanese) (Kyoto, 1999). Surikaisekikenkyusho Kokyuroku No. 1163 (2000), 85--98.

\bibitem{Ush} --------, {\em A volume formula for generalized hyperbolic tetrahedra}. Non-Euclidean geometries, 249--265, Math. Appl. (N. Y.), 581, Springer, New York, 2006.

\bibitem{W} Weiss, Hatmut, {\em The deformation theory of hyperbolic cone-3-manifolds with cone angles less than $2\pi$}. Geom. Topol. 17 (2013), no. 1, 329--367.











\end{thebibliography}
\end{document}